\documentclass[12pt]{amsart}
\usepackage[english]{babel}
\usepackage[utf8]{inputenc}
\usepackage[centering, margin={1in, 0.5in}, includeheadfoot]{geometry}
\usepackage{mathtools}
\usepackage[normalem]{ulem}
\usepackage{array}
\usepackage{tabularx}
\usepackage{graphicx}
\usepackage{amsmath}
\usepackage{amssymb}
\usepackage{amsthm}
\usepackage{enumerate}
\usepackage{color}
\usepackage{xcolor}
\definecolor{myred}{rgb}{0.2,0,0}
\definecolor{myblue}{rgb}{0,0,0.6}
\definecolor{mygreen}{rgb}{0,0.2,0}
\usepackage[colorlinks=true,linkcolor=myred,citecolor=myblue,urlcolor=mygreen]{hyperref}
\usepackage{enumitem}
\usepackage{bm}
\usepackage{bbm}

\usepackage{scalerel}

\numberwithin{equation}{section}

\allowdisplaybreaks

\newcommand{\Z}{\mathbb{Z}}
\newcommand{\R}{\mathbb{R}}

\newcommand{\comment}[1]{}

\DeclareMathOperator{\eps}{\epsilon}

\newcommand{\norm}[1]{\left\| #1 \right\|}

\renewcommand{\le}{\operatorname{\leqslant}}
\renewcommand{\ge}{\operatorname{\geqslant}}

\newcommand{\Id}{\mathbbm{1}}

\newtheorem{theorem}{Theorem}[section]
\newtheorem{lemma}[theorem]{Lemma}
\newtheorem{corollary}[theorem]{Corollary}
\newtheorem{proposition}[theorem]{Proposition}

\theoremstyle{definition}

\newtheorem*{definition*}{Definition} 

\theoremstyle{remark}

\begin{document}
	\selectlanguage{english}
    
\title{Pair Correlation of $\alpha n^{\theta}$ for Random $\theta$}

\author{Andrei Shubin}
\address{Institute of Analysis and Number Theory, TU Graz, Austria}
\email{andrei.shubin@tugraz.at}

\maketitle

\begin{abstract}

For fixed $\alpha>0$, we show that the sequence $\{\alpha n^{\theta}\}$ has Poissonian pair correlation for Lebesgue-almost all $\theta \in (0,\frac{3}{5})\cup(3,\infty)$. This improves a result of Technau and Yesha, who proved the same for almost all $\theta>7$.

The approach of Technau and Yesha was based on a repulsion principle, which roughly allows one to estimate the variance of the pair correlation function using the fourth derivative of the phase. In our approach, we split the $\theta$-integration in the variance into many short intervals and show that most of the integrals can be estimated using the first derivative. The problem is then reduced to several counting estimates, which we prove using moments of the Riemann zeta function and exponent pairs.

\end{abstract}

\section{Introduction}

A sequence $(x_n)_{n\in\mathbb{N}}$ of real numbers is said to be \emph{uniformly distributed modulo $1$} if for every interval $(a,b)\subseteq [0,1]$ one has
\[
    \frac{1}{N}\# \big\{1\le n\le N:\{x_n\}\in (a,b) \big\} \to b-a
    \qquad \text{as } N\to\infty.
\]
By classical works of Weyl, Fej\'er, and Csillag~\cite{Weyl, Csillag, KN74}, the sequences $(\alpha n^{\theta})_{n}$ are uniformly distributed modulo $1$ whenever $\alpha\neq 0$ and $\theta>0$ is non-integer, and whenever $\theta>0$ and~$\alpha$ is irrational.

The study of finer-scale statistics of such sequences was initiated by Rudnick and Sarnak~\cite{RS98}. A sequence $(x_n)_{n\in\mathbb{N}}$ is said to have \emph{Poissonian pair correlation} (PPC) if for every fixed $s>0$,
\[
    \frac{1}{N}\#\Bigl\{1\le m\neq n\le N:\|x_m-x_n\|\le \frac{s}{N}\Bigr\}\to 2s
    \qquad \text{as } N\to\infty.
\]
This property implies uniform distribution~\cite{ALP, GL}, and for i.i.d.\ random variables uniformly distributed on $[0,1)$ it holds almost surely. It is natural to conjecture that PPC should hold for the sequences $(\alpha n^{\theta})_n$ whenever $\alpha\neq 0$, $\theta>0$, $\theta\notin \mathbb{Z}$, and $\theta\neq \frac12$. For integer powers $\theta \ge 2$, it is conjectured under additional diophantine conditions on $\alpha$~\cite{Rudnick_Sarnak_Zaharescu}. The case of integer powers is particularly interesting in view of the Berry--Tabor conjecture~\cite{Berry_Tabor}.

At present, PPC is known for the sequences $(\alpha n^{\theta})_n$ for every $\alpha>0$ and $\theta<\frac{43}{117}$, by the work of Radziwi\l\l\ and the author~\cite{Shubin_Radziwill} and the earlier work of Lutsko, Sourmelidis, and Technau~\cite{LST}. Higher-order correlations for $(\alpha n^{\theta})_n$ are also known to be Poissonian with the admissible range for $\theta$ shrinking to zero as the order of correlation grows~\cite{LT_higher}. PPC was also established for the sequence $(\sqrt{n})_n$, after removing the squares, by El-Baz, Marklof, and Vinogradov~\cite{EMV}, although the gap distribution of $(\sqrt{n})_n$ is non-standard~\cite{Elk_McM, Browning_Vinogradov, Radziwill_Technau_sqrt} and higher-order correlations do not exist. Examples in which all correlations are known to be Poissonian include slowly growing sequences $(\alpha(\log n)^A)_n$ with $\alpha>0$ and $A>1$, due to Lutsko and Technau~\cite{Lutsko_Tech_log}, and sequences $(\alpha n)_n$ along rough numbers with badly approximable $\alpha$, due to a recent work of Hauke~\cite{Hauke_bad_rough}. We note that the full Kronecker sequence $(\alpha n)_n$ does not have PPC, by the three-gap theorem (see~\cite{LS20}).

Another direction of research concerns \emph{metric} pair correlation (MPPC), where one fixes a sequence $(a_n)_n$ and asks whether $(\alpha a_n)_n$ has PPC for Lebesgue-almost all $\alpha$. In this setting the problem is often more tractable for rapidly growing sequences $(a_n)_n$. Rudnick and Sarnak~\cite{RS98} proved MPPC for $(\alpha n^d)_n$ with integer $d\ge2$; in the quadratic case, alternative proofs were later given by Marklof and Str\"ombergsson~\cite{Mark_Strom} and by Heath-Brown~\cite{HB2010}. MPPC is also known for all real-valued lacunary sequences $(a_n)_n$~\cite{Rudnick_Zaharescu_lac, Rudnick_Technau_lac}. A general sufficient condition for MPPC in terms of a counting problem was obtained by Rudnick and Technau~\cite{Rudnick_Technau_lac}. Aistleitner, El-Baz, and Munsch~\cite{AEBM} introduced a different sufficient condition, based on estimates for moments of the Riemann zeta function and a variant of additive energy (their condition was recently sharpened by Kerr and Wang~\cite{Kerr_Wang}), and showed MPPC for $(\alpha n^{\theta})_n$ for every fixed $\theta>1$. Rudnick and Technau~\cite{RT} complemented this by proving MPPC for all $0<\theta<1$. Both results~\cite{AEBM, RT} rely on the energy bound for $(n^{\theta})_n$ due to Robert and Sargos~\cite{Robert-Sargos}. Very recently, higher-order correlations of $(\alpha n^d)_n$ were established for almost all $\alpha$ when $d$ grows with the order~\cite{LRT26}.

In this work we consider a different metric problem, in which $\alpha>0$ is fixed and $\theta$ is randomized. On a technical level, this problem appears to be more difficult than the corresponding $\alpha$-metric problem, because the variance integral with respect to $\theta$ is harder to control: the phase is no longer linear in the parameter and may have intricate stationary points structure that is difficult to handle. A first result in this direction was obtained by Technau and Yesha~\cite{Technau_Yesha}, who proved that $(n^{\theta})_n$ has PPC for Lebesgue-almost all $\theta>7$ and established higher-order correlations for sufficiently large $\theta$ depending on the order, using a \emph{repulsion principle}. Their argument extends directly to $(\alpha n^{\theta})_n$ for any fixed $\alpha>0$. A similar approach was later used by Aistleitner, Baker, Technau, and Yesha~\cite{ABTY} to prove Poissonian correlations of all orders for a broad class of super-polynomial sequences of the form $(\exp(\beta a_n))_n$ for almost all $\beta>0$ (as a consequence, the same was established for the sequences $(\beta^n)_n$ for almost all $\beta > 1$, refining the old result of Koksma~\cite{Koksma_alpha_n}).

Going back to $(\alpha n^{\theta})_n$, it seems that, as in the $\alpha$-metric case, the $\theta$-metric problem becomes harder near the transition point $\theta=1$. Moreover, Poisson summation reveals a certain symmetry between the regimes $\theta>1$ and $0<\theta<1$. Our main result is the following.

\begin{theorem} \label{main_thm}
Let $\alpha>0$ be fixed. Then the sequence $(\alpha n^{\theta})_{n\ge1}$ has Poissonian pair correlation for Lebesgue-almost all $\theta>3$ and for Lebesgue-almost all $\theta\in (0,\frac{3}{5})$.
\end{theorem}

Our approaches for large and small $\theta$ are quite similar. Triple Poisson summation applied to the pair correlation function lifts the problem from the regime $\theta>1$ to the regime $\theta<1$. In fact, the symmetry between them is better seen after only a double Poisson summation. Then, for example, by applying the repulsion principle of Technau and Yesha on both sides, one could extend the result of~\cite{Technau_Yesha} to almost all $\theta\in (0,\frac17)\cup(7,\infty)$, which is of course no longer interesting in view of~\cite{LST, Shubin_Radziwill}. 

We speculate that, if one could efficiently exploit the second derivative test instead of the fourth one (coming from the repulsion principle), then one could similarly reach the range $\theta\in (0,\frac13)\cup(3,\infty)$ (i.e. by showing that, for most oscillatory integrals appearing in the variance, the second derivative of the phase is large, where ``most'' needs to be quantified). Furthermore, if one could use only the first derivative test, which would mean that stationary points are absent in most integrals, one could obtain the full range $\theta\in(0,\infty)$. Our method, based on very short integrals on which stationary points do not play a role, essentially recovers the strength of the second derivative test. The reason we obtain $\frac35$ instead of $\frac13$ is that the third Poisson summation partially breaks the symmetry and yields additional savings. Interestingly, the exponential-sum and moment estimates required for $\theta<\frac35$ are milder than those required for the dual sums and moments when $\theta>3$, where we come close to using the sharpest bounds currently available. We note that the thresholds $3$ and $\frac{3}{5}$ themselves are independent of these bounds and seem to be the actual limitation of our method: they arise from the number of short integrals in the decomposition. To go beyond them, one would likely have to work with longer integration ranges and take the stationary points into account.

We remark finally that Technau and Yesha also established Poissonian $m$-point correlations for almost all $\theta>4m^2-4m-1$. The present method would likely allow one to improve these ranges as well, but we do not pursue this here.

\subsection*{Plan of the paper}

In Section~2 we provide an outline of the proof of Theorem~\ref{main_thm}. In Section~3 we collect the necessary auxiliary results on exponential sums and moments. Sections~4 and~5 contain the main parts of the proof for large and small $\theta$, respectively. In Sections~6 and~7 we prove the required counting estimates.

\subsection*{Notation}

We write $e(x):=e^{2\pi i x}$.
The relations $f(x)\ll g(x)$ and $f(x)=O(g(x))$ mean that $|f(x)|\le Cg(x)$ for some constant $C>0$ and all sufficiently large $x$, where $g(x)>0$. The notation $f(x)=o(g(x))$, with $g(x)>0$, means that $f(x)/g(x)\to 0$ as $x\to\infty$.

For real-valued functions $f$ and $g$, we write $f(x)\asymp g(x)$ if there exist constants $c,C>0$ such that
\[
    c|g(x)| \le |f(x)| \le C|g(x)|
\]
for all sufficiently large $x$, and moreover $f(x)$ and $g(x)$ have the same sign for all sufficiently large $x$. We write $f(x)\sim g(x)$ if $g(x)\neq 0$ for all sufficiently large $x$ and $f(x)/g(x)\to 1$ as $x\to\infty$. We write $\norm{x}$ for the distance from $x$ to the nearest integer.

For brevity, the ranges in multiple sums, such as $\sum_{K,Y_1,Y_2,R}$, are often omitted. Likewise, we write
\[
    \sum_{a\asymp A}
\]
to denote a sum over a range of the form $c_1A<a\le c_2A$, where the positive constants $c_1,c_2$ may vary from line to line, but are not important for the argument. Similarly, an arbitrarily small $\eps>0$ may vary from line to line in some proofs.

\subsection*{Acknowledgements}

This research was funded in whole or in part by the Austrian Science Fund (FWF). Throughout this work, the author was supported by the grants 10.55776/ESP531, 10.55776/I5554, and the FWF--ANR project Arithrand (I 4945-N and ANR-20-CE91-0006). In addition, part of this work was supported by the Swedish Research Council under grant no.~2021-06594 while the author was in residence at the Institut Mittag-Leffler in Djursholm, Sweden, during the analytic number theory program in 2024. I thank Christoph Aistleitner, Maksym Radziwi\l\l, Niclas Technau, Athanasios Sourmelidis, Eduard Stefanescu, Agamemnon Zafeiropoulos, and Manuel Hauke for helpful conversations.

\section{Proof outline}

\subsection*{Reduction to oscillatory integrals}

Let $f \in C_c^\infty(\R)$ be even and non-negative. By a standard argument, it is enough to prove the convergence of the smoothed pair correlation function
\[
    R_{f,N}(\theta) :=
    \frac{1}{N}\sum_{1\le m\neq n\le N} \sum_{k\in\Z}
    f \Big( N (\alpha m^{\theta} - \alpha n^{\theta} + k) \Big)
    \longrightarrow \int_{\R} f(x)\,dx
\]
for almost all $\theta$ in the relevant range. Applying Poisson summation and separating the zero frequency, this reduces to showing that
\begin{equation} \label{main_criterion}
    \frac{1}{N^2}\sum_{1\le k\le N^{1+\varepsilon}}
    \widehat{f} \Big(\frac{k}{N}\Big)
    \sum_{1\le m \neq n\le N} e \big(\alpha k (m^{\theta} -  n^\theta)\big) = o(1)
\end{equation}
for almost all $\theta$.

Next, fix an interval $I=(\theta_A,\theta_B)$ contained in either $(3,\infty)$ or $(0,\frac35)$ and consider the variance
\[
    V_f(N;I) := \int_I \Big|
    R_{f,N}(\theta)-\int_{\R}f(x)\,dx
    \Big|^2 \,d\theta.
\]
A bound of the form
\begin{equation} \label{main_goal}
    V_f(N;I)\ll N^{-\delta}
\end{equation}
with some $\delta>0$ is sufficient for almost-everywhere convergence. Indeed, if we choose a subsequence $N_j=\lfloor j^A\rfloor$ with $A\delta>1$, then~\eqref{main_goal} implies
\[
    \sum_{j=1}^\infty V_f(N_j;I)<\infty,
\]
and Chebyshev's inequality together with the first Borel--Cantelli lemma yields
\[
    R_{f,N_j}(\theta)\longrightarrow \int_{\R}f(x)\,dx
\]
for almost all $\theta\in I$. Since $N_{j+1}/N_j\to 1$, one may then pass from the subsequence to all $N$ by a standard interpolation argument (see~\cite[Section~7]{Technau_Yesha} for the details). Thus, it is enough to establish~\eqref{main_goal}.

In this section, we sketch the proof of Theorem~1.1. For simplicity, we omit various $\eps$-powers, logarithmic factors, and smooth weights, and restrict all sums to the most critical dyadic ranges. The full argument is given in Sections~4--7.

By~\eqref{main_criterion}, one needs to show roughly
\begin{equation} \label{to_show_roughly}
    \frac{1}{N^4} \sum_{k_1, k_2 \asymp N}
    \sum_{\substack{m_1, n_1, m_2, n_2 \asymp N \\ n_i - m_j \asymp N \\ m_1 - m_2 \asymp N \\ n_1 - n_2 \asymp N}}
    \int_{\theta_A}^{\theta_B}
    e\Big( \alpha k_1 (n_1^{\theta} - m_1^\theta) - \alpha k_2 (n_2^{\theta} - m_2^\theta) \Big) d\theta \ll N^{-\delta}.
\end{equation}
Thus the whole problem is reduced to bounding oscillatory integrals in $\theta$.

\subsection*{Repulsion principle}

The main technical input in~\cite{Technau_Yesha} is van der Corput's lemma.

\begin{lemma} \label{Corput_lemma}
    Let $f: \R \to \R$ be a smooth function. Suppose that for some fixed $\ell \ge 1$ one has
    $|f^{(\ell)}(x)| \ge \lambda > 0$ on an interval $I$, and if $\ell=1$, assume additionally that $f'$ is monotone on $I$. Then
    \[
        \Big| \int_I e\big( f(x) \big)\,dx \Big| \ll_\ell \lambda^{-1/\ell}.
    \]
\end{lemma}

Denote the phase in~\eqref{to_show_roughly} by
\[
    E (\theta) := \alpha k_1 \big(n_1^\theta-m_1^\theta\big) - \alpha k_2\big(n_2^\theta-m_2^\theta\big).
\]
Its derivatives are
\begin{equation} \label{E_der}
    E^{(\ell)} (\theta)
    = \alpha k_1 \big( n_1^\theta (\log n_1)^{\ell}-m_1^\theta (\log m_1)^{\ell}\big)
    - \alpha k_2 \big(n_2^\theta (\log n_2)^{\ell}-m_2^\theta (\log m_2)^{\ell} \big).
\end{equation}

The repulsion principle of Technau and Yesha is based on the fact that the logarithmic factors appearing in $E^{(\ell)}(\theta)$ for $\ell=1,2,3,4$ form a Vandermonde matrix. This implies that among the first four derivatives of the phase, at least one must be large. Heuristically, in the most critical dyadic range,
\[
    \max_{1\le \ell \le 4} |E^{(\ell)}(\theta)| \gg N^{1+\theta}
\]
for all $\theta \in (\theta_A,\theta_B)$ and for all tuples occurring in~\eqref{to_show_roughly}. After partitioning $(\theta_A,\theta_B)$ into $O(1)$ subintervals if necessary (since there are $O(1)$ sign changes of each derivative), one may therefore apply Lemma~\ref{Corput_lemma} with $\ell\le 4$, and in the worst case with $\ell=4$. This gives
\[
    V_f (N; (\theta_A, \theta_B))
    \ll \frac{1}{N^4}
    \sum_{k_1, k_2 \asymp N}
    \sum_{\substack{m_1, n_1, m_2, n_2 \asymp N \\ n_i - m_j \asymp N \\ m_1 - m_2 \asymp N \\ n_1 - n_2 \asymp N}}
    \big( N^{-1-\theta_A} \big)^{1/4}
    \ll N^{-4+6-1/4-\theta_A/4}.
\]
Thus one gets a power saving as soon as $\theta_A>7$.

\subsection*{New approach: short intervals}

Our approach begins by splitting the integral in~\eqref{to_show_roughly} into many short intervals. Namely,
\[
    V_f (N; (\theta_A, \theta_B))
    = \sum_{i=0}^{H-1} V_f^{(i)} (N),
\]
where
\begin{gather*}
    V_f^{(i)} (N)
    \approx \frac{1}{N^4}
    \sum_{k_1, k_2 \asymp N}
    \sum_{\substack{m_1, n_1, m_2, n_2 \asymp N \\ n_i - m_j \asymp N \\ m_1 - m_2 \asymp N \\ n_1 - n_2 \asymp N}}
    \int_{\theta_i}^{\theta_i + \frac{1}{H}} e\big( E(\theta) \big)\, d\theta, \\
    \theta_i = \theta_A + \frac{i (\theta_B - \theta_A)}{H}.
\end{gather*}
We choose
\[
    H := N^{\frac{1+\theta_B}{2} + \eps}.
\]
Thus it is enough to prove
\[
    V_f^{(i)} (N) \ll H^{-1} N^{-\delta}
\]
for every $i$.

Applying Taylor expansion at $\theta=\theta_i$, we get
\[
    E(\theta)=E(\theta_i)+(\theta-\theta_i)E^{(1)}(\theta_i) +O\Big(\frac{1}{H^2}\max |E^{(2)}(\theta)|\Big),
\] where $E^{(1)} (\theta)$ and $E^{(2)} (\theta)$ are the first and second derivatives of $E(\theta)$ given by~\eqref{E_der}. By our choice of $H$, the quadratic error term is of size $o(1)$ and may therefore be absorbed. Substituting $u=\theta-\theta_i$, we arrive at
\begin{multline} \label{final_shape_sketch}
    V_f^{(i)} (N)
    \ll
    \frac{1}{N^4}
    \sum_{k_1, k_2 \asymp N}
    \sum_{\substack{m_1, n_1, m_2, n_2 \asymp N \\ n_i - m_j \asymp N \\ m_1 - m_2 \asymp N \\ n_1 - n_2 \asymp N}}
    \bigg| \int_0^{1/H} e\big( u E^{(1)}(\theta_i) \big)\,du \bigg| \ll \\
    \frac{1}{N^4}
    \sum_{k_1, k_2 \asymp N}
    \sum_{\substack{m_1, n_1, m_2, n_2 \asymp N \\ n_i - m_j \asymp N \\ m_1 - m_2 \asymp N \\ n_1 - n_2 \asymp N}}
    \min\Big( \frac{1}{H}, \frac{1}{|E^{(1)}(\theta_i)|} \Big).
\end{multline} Note that with any smaller choice of $H$ (longer integrals), one could not ignore the quadratic term $u^2 E^{(2)}(\theta_i)$ and therefore could not always get $|E^{(1)}(\theta_i)|^{-1}$ as the second upper bound.

\subsection*{A counting problem}

We now split the 6-tuples $(k_1, k_2, m_1, m_2, n_1, n_2)$ in~\eqref{final_shape_sketch} into dyadic level sets according to the value of $|E^{(1)}(\theta_i)|$. Let
\[
    J_{\Delta} :=
    \Big\{ k_1, k_2, m_1, n_1, m_2, n_2 \asymp N:
    \ |E^{(1)} (\theta_i)| \asymp \Delta N^{1+\theta_i} \Big\},
\]
where $\Delta$ runs over the values $1,\frac12,\frac14,\dots$. Our goal is then to prove the bound
\begin{equation} \label{desired_bound_sketch}
    \# \Big\{ k_1, k_2, m_1, n_1, m_2, n_2 \asymp N:
    \ |E^{(1)} (\theta_i)| \le \Delta N^{1+\theta_i} \Big\}
    \ll N^{4-\kappa} + \Delta N^6
\end{equation}
with some $\kappa>0$. This would imply the same upper bound for $J_\Delta$, and therefore
\[
    V_f^{(i)} (N) \ll
    \sum_{\substack{\Delta \text{ dyadic} \\ 1 \ge \Delta \gg N^{-2-\kappa}}}
    \frac{1}{N^4}
    \big( N^{4-\kappa} + \Delta N^6 \big)
    \min\Big( \frac{1}{H}, \frac{1}{\Delta N^{1+\theta_i}} \Big) + O(N^{-\kappa} H^{-1}).
\]
A short calculation gives
\[
    V_f^{(i)} (N) \ll
    \frac{N^{-\kappa}\log N}{H} + \frac{\log N}{N^{\theta_i-1}},
\]
which is $\ll H^{-1}N^{-\delta}$ as soon as $\kappa > \delta$ and $\theta_i>3+2\eps+2\delta$, say. Finally, take $\eps,\delta\to 0$. Thus, the problem is reduced to proving~\eqref{desired_bound_sketch}.

Aistleitner, El-Baz, and Munsch~\cite{AEBM} decomposed a general counting problem of this form into what can be called a ``zeta part'' and an ``additive-energy part''. The condition
\[
    |E^{(1)}(\theta_i)|\le \Delta N^{1+\theta}
\]
can be written in the multiplicative form roughly as
\[
    \bigg| \log \frac{k_1 (n_1^{\theta_i} \log n_1 - m_1^{\theta_i} \log m_1)}
    {k_2 (n_2^{\theta_i} \log n_2 - m_2^{\theta_i} \log m_2)} \bigg| \le \Delta.
\] After a standard Fourier decomposition, one roughly gets
\begin{multline} \label{main_expression}
    \# J_{\Delta}
    \approx \Delta \int_0^{1/\Delta} \bigg| \sum_{k_1, k_2 \asymp N} \sum_{\substack{m_1, m_2 \asymp N \\ n_1, n_2 \asymp N \\ n_i > m_i}} e\Big( t \log \frac{k_1 (n_1^{\theta_i} \log n_1 - m_1^{\theta_i} \log m_1)}
    {k_2 (n_2^{\theta_i} \log n_2 - m_2^{\theta_i} \log m_2)} \Big) \bigg| dt \ll \\
    \Delta \int_0^{1/\Delta}
    \Big| \sum_{k \asymp N} k^{it} \Big|^2
    \Big| \sum_{\substack{m, n \asymp N \\ n > m}} e\Big( t \log \big( n^{\theta_i} \log n - m^{\theta_i} \log m \big)  \Big) \Big|^2 dt.
\end{multline} We note that one roughly has
\[
    \Delta \int_0^{1/\Delta}
    \Big| \sum_{\substack{m, n \asymp N \\ n > m}} e\Big( t \log \big( n^{\theta_i} \log n - m^{\theta_i} \log m \big) \Big) \Big|^2 dt \ll \Delta \int_0^{1/\Delta}
    \Big| \sum_{m \asymp N} e\Big( t \frac{m^{\theta_i} \log m}{N^{\theta_i} \log N} \Big) \Big|^4 dt
\] since both expressions can be bounded by similar counting quantities. Therefore, if one had the optimal additive-energy bound
\begin{multline} \label{desired_energy_bound}
    \Delta \int_0^{1/\Delta} \Big| \sum_{m \asymp N} e\Big( t \frac{m^{\theta_i} \log m}{N^{\theta_i} \log N} \Big) \Big|^4 dt \approx
    \# \Big\{ m_1, \ldots, m_4 \asymp N: \\ 
    \big| m_1^{\theta_i} \log m_1 -  m_2^{\theta_i} \log m_2 + m_3^{\theta_i} \log m_3 - m_4^{\theta_i} \log m_4 \big| < \Delta N^{\theta_i} \log N \Big\} \ll
    N^{\eps} \big( N^2 + \Delta N^4 \big), 
\end{multline}
then~\eqref{desired_bound_sketch} would follow immediately from any power saving bound on the zeta sum. An analogous counting problem in~\cite{AEBM, RT} goes without logarithmic factors, and there the same upper bound (as in~\eqref{desired_energy_bound}) follows from well-known results of Robert and Sargos~\cite{Robert-Sargos} and Huang~\cite{Huang}. Although the bound~\eqref{desired_energy_bound} should likely hold in this case as well, the approaches from~\cite{Robert-Sargos, Huang} do not seem to generalize directly. 

Instead, we follow the strategy in~\cite{AEBM}. First, the trivial bound gives
\[
    \Delta \int_0^1 \ldots dt \ll \Delta \int_0^1 N^6 dt \ll \Delta N^6.
\] Second, applying the Kusmin--Landau inequality, we get
\[
    \Delta \int_1^N \ldots dt \ll \Delta \int_1^N \Big( \frac{N}{t}  \Big)^2 \cdot N^4 dt \ll \Delta N^6. 
\] 

We split the remaining part of the integral in~\eqref{main_expression} into dyadic intervals $t\in[T,2T]$. 
Assume that $N < T \le N^{2+\eps}$.
Applying Hölder's inequality, we obtain
\begin{multline} \label{after_holder_sketch}
    \# J_{\Delta} \ll \Delta
    \bigg( \int_T^{2T} \Big| \sum_{k \asymp N} k^{it} \Big|^8 dt \bigg)^{1/4}
    \cdot \\
    \bigg(
    \max_{t \in [T, 2T]}
    \Big| \sum_{m \asymp N} \sum_{n \asymp N} e\Big( t \log (n^{\theta_i} \log n - m^{\theta_i} \log m) \Big) \Big|^{2/3}
    \bigg)^{3/4} \cdot 
    \bigg( \int_T^{2T} \Big| \sum_{m \asymp N} e\Big( t \frac{m^{\theta_i} \log m}{N^{\theta_i} \log N} \Big) \Big|^4 dt \bigg)^{3/4}.
\end{multline}
The first factor is controlled by the eighth moment of the zeta function, and can be estimated by
\[
    \Delta \big( N^4 T^{\frac{3}{2}} \big)^{\frac{1}{4}}
\] by work of Ingham, Heath-Brown, and Ivi{\'c}. The second factor is estimated by
\[
    \big( N \cdot N^{\frac{1}{2}} T^{\frac{13}{84}} \big)^{\frac{2}{3} \cdot \frac{3}{4}}
\] using Bourgain's exponent pair~\cite{Bourgain16}. The third factor plays the role of an additive-energy term; here we mimic part of the argument of Robert and Sargos~\cite{Robert-Sargos}, but stop at a scale where the available exponent-pair bounds already give sufficient cancellation: the bound we obtain for the third factor in~\eqref{after_holder_sketch} is
\[
    N^{\eps} \big( T N^{2.45} + N^4\big)^{\frac{3}{4}}.
\] Combining the last three estimates yields
\[
    \#J_\Delta \ll \Delta N^{5.993} \le \Delta N^6.
\] Finally, if $T > N^{2+\eps}$, we have $\Delta^{-1} > N^{2+\eps}$, and so by monotonicity of $\# J_{\Delta}$, we get
\[
    \# J_{\Delta} \ll \# J_{N^{-2-\eps}} \ll N^{5.993 - 2 - \eps} \ll N^{4-\kappa}.
\] Together, these bounds give~\eqref{desired_bound_sketch}, as desired.

We stress again that the threshold $\theta>3$ does not come from the quality of the bounds used in~\eqref{after_holder_sketch}; it already appears at the previous step from the choice of $H$, which is essentially dictated by our aim of applying Taylor expansion to $E(\theta)$.

\subsection*{Small $\theta$ regime}

We now consider the range $0<\theta<1$. Applying Poisson summation in $m,n,k$ to the sum in~\eqref{main_criterion}, as in~\cite{LST, Shubin_Radziwill}, leads heuristically to
\[
    S(N) := \frac{1}{N^2} \sum_{k \asymp N}
    \sum_{\substack{m,n \asymp N \\ n-m \asymp N}}
    e\big( \alpha k (m^{\theta} - n^{\theta}) \big)
    \approx N^{-\frac{1}{2} - \frac{3\theta}{2}}
    \sum_{\ell \asymp N^{\theta}}
    \sum_{m,n \asymp N^{\theta}}
    e\big( \ell^{\frac{1}{\theta}} \eta^{1-\frac{1}{\theta}} \big),
\]
where
\[
    \eta(\theta) := m^{-\frac{\theta}{1-\theta}} - n^{-\frac{\theta}{1-\theta}}.
\]
Thus PPC follows from $S(N)=o(1)$, and one is led to estimating a dual variance, which is of the form
\[
    V_f (N; (\theta_A, \theta_B)) :=
    \int_{\theta_A}^{\theta_B}
    \frac{1}{N^{1+3\theta}}
    \sum_{\ell_1, \ell_2 \asymp N^{\theta}}
    \sum_{\eta_1, \eta_2}
    e\big( \ell_1^{\frac{1}{\theta}} \eta_1^{1-\frac{1}{\theta}} - \ell_2^{\frac{1}{\theta}} \eta_2^{1-\frac{1}{\theta}} \big)
    d\theta.
\]

We again split the $\theta$-integral into short pieces of length $H^{-1}$ with
\[
    H:=N^{\frac{1+\theta_B}{2} + \eps},
\]
and repeat the same argument as before. This yields
\[
    V_f^{(i)} (N) \ll
    \frac{1}{N^{1+3\theta_i}}
    \sum_{\ell_1,\ell_2 \asymp N^{\theta_i}}
    \sum_{\eta_1,\eta_2}
    \min\Big( \frac{1}{H}, \frac{1}{|E^{(1)}(\theta_i)|} \Big),
\]
where now
\[
    E^{(1)}(\theta) =   F(\ell_1,\eta_1,\theta)\,\ell_1^{\frac{1}{\theta}}\eta_1^{1-\frac{1}{\theta}} -
    F(\ell_2,\eta_2,\theta)\,\ell_2^{\frac{1}{\theta}}\eta_2^{1-\frac{1}{\theta}}
\] for a smooth function \(F\) of size roughly \(\log N\).

Thus everything reduces again to a counting problem. In this setting, one expects the bound
\begin{equation} \label{strong_counting}
    \#J_\Delta = \# \Big\{ k_1, k_2, m_1, n_1, m_2, n_2 \asymp N^{\theta_i}: \ |E^{(1)}(\theta_i)| \asymp \Delta N^{1+\theta_i}  \Big\} \ll N^{4\theta_i - \eps} + \Delta N^{6\theta_i}.
\end{equation} We will nevertheless obtain only a weaker bound, roughly
\begin{equation} \label{weak_counting}
    \ll N^{(4+\frac{2}{3}) \theta_i} + \Delta N^{6\theta_i},
\end{equation} which is the necessary minimum for our goal.
Substituting this into the analogue of~\eqref{final_shape_sketch}, we obtain
\[
    V_f^{(i)} (N) \ll
    \sum_{\substack{\Delta \text{ dyadic} \\ 1 \ge \Delta \gg N^{-2\theta_i-\eps}}}
    \frac{1}{N^{1+3\theta_i}}
    \Big( N^{(4+2/3)\theta_i} + \Delta N^{6\theta_i} \Big)
    \min \Big( \frac{1}{H}, \frac{1}{\Delta N^{1+\theta_i}} \Big),
\]
and hence
\[
    V_f^{(i)} (N) \ll
    \frac{N^{(5/3)\theta_i-1}\log N}{H} +
    \frac{\log N}{N^{2-2\theta_i}}.
\]
This is $\ll H^{-1}N^{-\delta}$ provided \(\theta_i \le \theta_B < \frac35 - \delta - \eps\). Finally, again, take $\delta, \eps \to 0$. 

Note that the threshold $\frac{3}{5}$ comes from both terms in the counting bound, and since the term $\Delta N^{6\theta_i}$ cannot be avoided, it again cannot be surpassed at this step by simply improving the first term in~\eqref{weak_counting} or~\eqref{strong_counting}.

\section{Auxiliary lemmas} \label{sec_aux_lemmas}

Here we provide the necessary technical statements used in the proof.

\begin{lemma} \label{Kusmin_Landau}
    Let $f:[a,b]\to \R$ be a $C^1$ function such that $f'$ is monotone on $[a,b]$. Assume moreover that
    \[
        \|f'(x)\| \ge \lambda >0
        \qquad \text{for all } x\in [a,b].
    \]
    Then
    \[
        \sum_{a<n\le b} e\big(f(n)\big) \ll \lambda^{-1}.
    \]
\end{lemma}

See, for example,~\cite[Corollary~8.11]{IK}.

\begin{lemma} \label{Corput}
    Let $f:[a,b]\to \R$ be a $C^2$ function such that
    \[
        |f''(x)| \asymp \lambda_2 >0
        \qquad \text{for all } x\in [a,b].
    \]
    Then
    \[
        \sum_{a<n\le b} e\big(f(n)\big)
        \ll
        (b-a)\lambda_2^{\frac{1}{2}}+\lambda_2^{-\frac{1}{2}}.
    \]
\end{lemma}

See~\cite[Corollary~8.13]{IK}.

\begin{lemma} \label{MVT_zeta}
    Let $T \ge N \ge 2$ and let $A \ge 2$. We have the bound
    \[
        \int_T^{2T} \Big| \sum_{N < n \le 2N} n^{it} \Big|^{2A} dt \ll_{A,\varepsilon} 
        \begin{cases}
        N^A T^{\frac12+\frac A4+\varepsilon}, & 2\le A<6,\\[1mm]
        N^A T^{\frac{3A+4}{11}+\varepsilon}, & 6\le A<\frac{89}{13},\\[1mm]
        N^A T^{\frac{35A+3}{108}+\varepsilon}, & \frac{89}{13}\le A<7,\\[1mm]
        N^A T^{\frac{9A+1}{28}+\varepsilon}, & A\ge 7.
    \end{cases}
    \]
    for any $\eps > 0$.
\end{lemma}

\begin{proof}
Let
\[
    P_N(t):=\sum_{N<n\le 2N} n^{it}, \qquad
    J_{2A} (N,T):=\int_T^{2T} |P_N(t)|^{2A}\,dt, \qquad
    I_{2A}(Y):=\int_0^Y \bigl|\zeta(\tfrac12+iu)\bigr|^{2A}\,du.
\] We will show that for any $\varepsilon>0$ one has
\[
    J_{2A}(N,T) \ll_{A,\varepsilon} T\log N + N^{2A} + N^A (TN)^{\eps} I_{2A}\!\big(C(T+N\log N)\big),
\] where $C>0$ is an absolute constant. Write
\[
    A(x,t):=\sum_{n\le x} n^{it}, \qquad P_N (t) = A(2N,t)-A(N,t).
\] By Perron formula,
\[
    A(x,t) = \frac{1}{2\pi i}\int_{\kappa-iU}^{\kappa+iU}
    \zeta(w-it)\frac{x^w}{w}\,dw + O\!\Big(\frac{x\log x}{U}+\log U + 1\Big).
\] We choose
\[
    \kappa:=1+\frac1{\log N}, \qquad U:=2\big(T+N\log N\big).
\] Thus
\[
    P_N (t) = \frac{1}{2\pi i}\int_{\kappa-iU}^{\kappa+iU}
    \zeta(w-it)\frac{(2N)^w - N^w}{w}\,dw +O\big(\log N\big).
\] 

We now shift the contour to $\Re w=\frac{1}{2}$. Since $T \le t \le 2T$ and $2T < U$, we pick up the pole at $w=1+it$: 
\[
    P_N (t) = \frac{(2N)^{1+it} - N^{1+it}}{1+it} + \frac{1}{2\pi i}\int_{1/2-iU}^{1/2+iU} \zeta(w-it)\frac{(2N)^w - N^w}{w}\,dw + H^{\pm}(N,t) + O(\log N),
\] where $H^\pm(N,t)$ denote the contribution of horizontal integrals along $\Im w = \pm U$. Next, we will show that $H^\pm(N,t) \ll 1$ uniformly for $T \le t \le 2T$. Indeed, on the horizontal sides we have $|\,\Im(w-it)\,|\asymp U$. For $\frac12\le \sigma\le 1$, the standard bound (see, e.g.,~\cite[Theorem~II.3.8]{Tenenbaum08}) implies
\[
    \zeta(\sigma+it)\ll_\varepsilon |t|^{\frac{1-\sigma}{3}+\varepsilon}.
\] Hence
\[
    \int_{1/2}^{1} \Big|\zeta(\sigma\pm iU-it)\frac{(2N)^\sigma - N^{\sigma}} {\sigma\pm iU}\Big|\,d\sigma \ll_\varepsilon
    \int_{1/2}^{1} N^\sigma U^{-1+\frac{1-\sigma}{3}+\varepsilon}\,d\sigma \ll 1,
\] since $U \gg N\log N$. On the short segment
$1 \le \sigma \le \kappa$, we use the classical bound
$\zeta(\sigma+it) \ll \log(|t|)$, obtaining
\[
    \int_1^\kappa \Big|\zeta(\sigma\pm iU-it)\frac{(2N)^{\sigma} - N^\sigma}{\sigma\pm iU}\Big|\,d\sigma \ll \int_1^\kappa N^\sigma U^{-1}\log U\,d\sigma \ll \frac{N\log U}{U}\ll 1.
\]

Writing $w=\frac12+iv$, we get
\[
    P_N (t) = \frac{(2N)^{1+it} - N^{1+it}}{1+it} + \frac{1}{2\pi}\int_{-U}^{U} \zeta\! \big(\tfrac12+i(v-t) \big) \frac{(2N)^{1/2+iv} - N^{1/2+iv}}{1/2+iv}\,dv + O(\log N).
\] Hence
\[
    P_N (t) \ll \frac{N}{1+t} + N^{\frac{1}{2}}\int_{-U}^{U} \frac{\big|\zeta(\tfrac12+i(v-t))\big|}{1+|v|}\,dv + \log N.
\]

Set
\[
    K_U(v):=\frac{\mathbbm{1}_{[-U,U]}(v)}{1+|v|}, \qquad f(v) := \big|\zeta(\tfrac12+i(v-t))\big|.
\] Then
\[
    |P_N(t)| \ll \frac{N}{1+t} + N^{\frac{1}{2}}\int_{\mathbb R} K_U(v)\, f(v)\,dv + \log N.
\]

By H\"older's inequality,
\[
    \Big(\int_{\mathbb R} K_U(v)f(v)\,dv\Big)^{2A} \le \Big(\int_{\mathbb R} K_U(v)\,dv\Big)^{2A-1} \int_{\mathbb R} K_U(v)f(v)^{2A}\,dv.
\] Since
\[
    \int_{\mathbb R} K_U(v)\,dv = \int_{-U}^{U}\frac{dv}{1+|v|} \ll \log U,
\] we deduce
\[
    |P_N(t)|^{2A} \ll_A \Big(\frac{N}{1+t}\Big)^{2A} + N^A(\log U)^{2A-1} \int_{-U}^{U} \frac{\bigl|\zeta(\tfrac12+i(v-t))\bigr|^{2A}}{1+|v|}\,dv + (\log N)^{2A}.
\] Integrating over $T \le t \le 2T$, and applying Fubini theorem, we get
\begin{multline*}
    J_{2A}(N,T) \ll_A \\
    \int_T^{2T} \Big(\frac{N}{1+t}\Big)^{2A}dt + N^A(\log U)^{2A-1} \int_{-U}^{U}\frac{1}{1+|v|} \Big(\int_T^{2T} \bigl|\zeta(\tfrac12+i(v-t))\bigr|^{2A}dt\Big)dv + T (\log N)^{2A}.
\end{multline*}

The first integral is $\ll_A N^{2A} T^{1-2A}$. For the inner integral in the second term, put $u:=v-t$. Then
\[
    \int_T^{2T} \bigl|\zeta(\tfrac12+i(v-t))\bigr|^{2A}dt = \int_{v-2T}^{v-T}\bigl|\zeta(\tfrac12+iu)\bigr|^{2A}du.
\]
Since $|v|\le U$ and $T \le t \le 2T < U$, the interval $[v-2T,v-T]$ is contained in $[-2U,2U]$, and therefore, using
\[
    \big|\zeta(\tfrac12-iu)\big|=\big|\zeta(\tfrac12+iu)\big|,
\] we obtain
\[
    \int_T^{2T} \bigl|\zeta(\tfrac12+i(v-t))\bigr|^{2A}dt \le 2I_{2A}(2U).
\] 

Hence
\[
    J_{2A}(N,T) \ll_A N^{2A} T^{1-2A} + N^A(\log U)^{2A} I_{2A}(2U) + T (\log N)^{2A}.
\]

Finally, since $U\asymp T+N\log N$, we conclude that
\[
    J_{2A}(N,T) \ll_{A,\varepsilon} T(\log N)^{2A} + N^{2A} T^{1-2A} + N^A (TN)^{\eps}I_{2A}\!\big(C(T+N\log N)\big)
\] for some absolute constant $C>0$. \\

Next, assume that
\[
    I_{2A}(U)\ll_\varepsilon U^{M(2A)+\varepsilon},
\] where $M(\cdot)$ is an admissible exponent in the bound for the moment of $\zeta(\tfrac12+it)$. Then we have
\[
    J_{2A}(N,T) \ll_{A,\varepsilon}
    T (\log N)^{2A} + N^{2A} T^{1-2A}+N^A (T+N)^{M(2A)+\varepsilon}.
\] 

From known results on the moments of the zeta function (see~\cite[Section~2.1]{TY24}) we have
\[
    M(B)\le
    \begin{cases}
        1+\frac{B-4}{8}, & 4\le B<12,\\[1mm]
        2+\frac{3(B-12)}{22}, & 12\le B<\frac{178}{13},\\[1mm]
        1+\frac{35(B-6)}{216}, & \frac{178}{13} \le B < 14,\\[1mm]
        1+\frac{9(B-6)}{56}, & B \ge 14.
\end{cases}
\]

Substituting $B=2A$ gives the desired bound.

\end{proof}

We remark that the lemma above clearly extends from $N < n \le 2N$ to any interval of the form $c_1 N < n \le c_2 N$.

\begin{lemma} \label{exp_pairs}
    Let $\theta>0$ be fixed. For $N\ge 2$, $T>cN$ for some $c > 0$, and $1\le r<N$, define
    \[
        \phi_1(x):=\frac{x^\theta \log x}{N^\theta \log N}, \qquad
        \phi_2(x):=\log\! \big(x^\theta \log x-(x-r)^\theta \log(x-r) \big),
        \qquad x \asymp N, \ \ r < 0.99x.
    \]
    Then for any $\eps>0$ one has
    \[
        \sum_{N<n\le 2N} e\big(T\phi_j(n)\big)
        \ll_{\theta,\eps}
        N^{\eps}
        \begin{cases}
            N^{\frac{187}{641}} T^{\frac{178}{641}},
            & \\[1mm]
            N^{\frac{1}{2}} T^{\frac{13}{84}},
            & 
        \end{cases}
    \]
    uniformly for $j=1,2$, and in the case $j=2$ uniformly in $r$.
\end{lemma}

\noindent
This is a standard consequence of exponent-pair theory. The first bound follows from Huxley's exponent pair $(\frac{178}{641}, \frac{365}{641})$, see \cite[Table~17.1]{Huxley96}; see also \cite[Section~3]{TTY25}. In the range $N^{1295/872}\le T\le N^{601/374}$ it is currently the strongest. The second bound follows from Bourgain's exponent pair $(\frac{13}{84}+\eps,\frac{55}{84}+\eps)$, see~\cite[Theorem~6]{Bourgain16}.

\begin{lemma}     \label{unbalanced_energy_bound}
    Let $\theta>1$ be fixed. For $N\ge 2$, $1\le M\le N$, and $T>0$, define
    \[
        A(v):=\sum_{n \asymp N} e\!\Big(v\frac{n^\theta\log n}{N^\theta\log N}\Big), \qquad
        B(v):=\sum_{m \asymp M} e\!\Big(v\frac{m^\theta\log m}{N^\theta\log N}\Big).
    \]
    Then for any $\eps>0$ one has the following two bounds.

    \smallskip

    \textit{(i) The first bound:}
    \[
        \int_0^T |A(v)|^2 |B(v)|^2\,dv
        \ll_{\theta,\eps}
        N^{\eps} \big( M^2N^2 + M^2NT \big).
    \]

    \smallskip

    \textit{(ii) Assuming additionally that $M \ge N^{\kappa}$ for some $\kappa > 0$ and $c_0 > 0$ is sufficiently small, the second bound:}
    \[
        \int_0^T |A(v)|^2 |B(v)|^2\,dv
        \ll_{\theta,\eps}
            N^{\eps} \Big( N^2M^2 + \dfrac{N^{2\theta}}{M^{2\theta-2}} \Big),
    \] if $0<T\le c_0 N^\theta / M^{\theta-1}$, and
    \[        
           \int_0^T |A(v)|^2 |B(v)|^2\,dv
        \ll_{\theta,\eps} N^{\eps} \Big( 1 + \frac{T}{N^{1.55}} \Big) \Big( N^2M^2
            + M^{\frac{374}{641}+\frac{356\theta}{641}} N^{3.411 - \frac{356\theta}{641}} \Big)
    \] otherwise.
\end{lemma}

\begin{proof}

    We first compute the second moment of the long sum \(A(v)\). Expanding the square, we have
    \[
        \int_X^{2X}|A(v)|^2\,dv =
        \sum_{n_1, n_2 \asymp N} \int_X^{2X} 
        e\!\Big(v\frac{n_1^\theta\log n_1-n_2^\theta\log n_2}{N^\theta\log N}\Big)\,dv.
    \]
    By the mean-value theorem,
    \[
        \Big|
        \frac{n_1^\theta\log n_1-n_2^\theta\log n_2}{N^\theta\log N}
        \Big|
        \asymp \frac{|n_1-n_2|}{N},
    \]
    and therefore
    \begin{equation} \label{A_second_moment_mixed}
        \int_X^{2X}|A(v)|^2\,dv
        \ll XN + N^2\log N.
    \end{equation}

    \vspace{2ex}
    \textit{Proof of the first bound.}
    
    First, if \(0<T<1\), then, using the trivial bounds, we get
    \begin{equation} \label{first_bound_trivial_range}
        \int_0^T |A(v)|^2 |B(v)|^2\,dv \ll N^2M^2.
    \end{equation}

    Next, split the remaining integral dyadically:
    \[
        \int_0^T |A(v)|^2|B(v)|^2\,dv
        \ll \int_0^1 |A(v)|^2|B(v)|^2\,dv +
        \sum_{\substack{X\text{ dyadic}\\ 1\le X\le T}}
        \int_X^{2X}|A(v)|^2|B(v)|^2\,dv.
    \]

    If \(1\le X\le c_1 N\) for sufficiently small $c_1 > 0$, then
    \[
        \frac{d}{dx}\Big(v\frac{x^\theta\log x}{N^\theta\log N}\Big) = v\frac{x^{\theta-1} (\theta \log x + 1)}{N^{\theta} \log N}  \asymp \frac{X}{N},
    \]
    and this derivative is monotone in \(x\). Hence Lemma~\ref{Kusmin_Landau} gives
    \[
        A(v)\ll \frac{N}{X}.
    \]
    Estimating the sum over $m \asymp M$ trivially, we obtain
    \[
        \int_X^{2X}|A(v)|^2|B(v)|^2\,dv
        \ll X\Big(\frac{N}{X}\Big)^2 M^2
        = \frac{N^2M^2}{X}.
    \]
    Summing over dyadic \(X\) with \(1\le X\le \min(T,N)\), we get
    \begin{equation} \label{first_bound_KL_range}
        \int_1^{\min(T,N)} |A(v)|^2|B(v)|^2\,dv \ll N^2 M^2.
    \end{equation}

    Finally, if \(X > c_1 N\), then we combine the trivial bound for the sum $m \asymp M$ with~\eqref{A_second_moment_mixed}:
    \[
        \int_X^{2X}|A(v)|^2|B(v)|^2\,dv \ll
        M^2 \int_X^{2X}|A(v)|^2\,dv \ll
        M^2 X N (\log N).
    \]
    Summing over dyadic \(X\) with \(N\le X\le T\), we obtain
    \begin{equation} \label{first_bound_second_mom_range}
        \int_N^T |A(v)|^2|B(v)|^2\,dv
        \ll M^2NT (\log N).
    \end{equation}
    Combining~\eqref{first_bound_trivial_range}, \eqref{first_bound_KL_range}, and~\eqref{first_bound_second_mom_range} proves the first bound.

    \vspace{2ex}
    
    \textit{Proof of the second bound.}

    The ranges \(0<v<1\) and $v \asymp X$ with \(1\le X\ll N\) are handled exactly as above, giving a contribution
    \[
        \ll N^{\eps} N^2M^2.
    \]

    We now consider dyadic blocks \(X\le v\le 2X\) with \(X\ge N\).

    \vspace{2ex}
    
    \textit{Range \(N\le X\le c_0 N^\theta/M^{\theta-1}\) with sufficiently small $c_0 > 0$.}
    When $m \asymp M$, we have
    \[
        \frac{d}{dm}\Big(v\frac{m^\theta\log m}{N^\theta\log N}\Big)
        \asymp \frac{X\,M^{\theta-1}}{N^\theta} < 1.
    \]
    Then Lemma~\ref{Kusmin_Landau} gives
    \[
        B(v)\ll \frac{N^\theta}{X\,M^{\theta-1}}.
    \]
    
    Therefore, by~\eqref{A_second_moment_mixed},
    \begin{multline*}
        \int_X^{2X}|A(v)|^2|B(v)|^2\,dv
        \ll \Big(\frac{N^\theta}{X\,M^{\theta-1}}\Big)^2
        \int_X^{2X}|A(v)|^2\,dv \ll \\
     \Big(\frac{N^\theta}{X\,M^{\theta-1}}\Big)^2
        \cdot XN (\log N) =
        \log N \frac{N^{2\theta+1}}{X\,M^{2\theta-2}}.
    \end{multline*}
    Summing over dyadic \(X\) with
    \[
        N\le X\le \min\Big(T,\frac{N^\theta}{M^{\theta-1}}\Big),
    \]
    gives
    \[
        \int_N^{\min(T,N^\theta/M^{\theta-1})}|A(v)|^2|B(v)|^2\,dv
        \ll \log N \frac{N^{2\theta}}{M^{2\theta-2}}.
    \]

    This proves the first part of the second bound.

    \vspace{2ex}
    
    \textit{Range \(c_0 N^\theta/M^{\theta-1}\le X\le N^{1.55}\).} Note that this range might be empty depending on $M,N,\theta$. Put
    \[
        T_M:=\frac{X\,M^\theta\log M}{N^\theta\log N}.
    \]
    In this range one has \(T_M\gg M\). Applying the first bound of Lemma~\ref{exp_pairs}, we obtain
    \[
        B(v) \ll M^{187/641+\eps}T_M^{178/641+\eps}.
    \]
    Hence
    \[
        B(v) \ll N^{-\frac{178\theta}{641}+\eps}
        M^{\frac{187}{641}+\frac{178\theta}{641}}
        X^{\frac{178}{641}+\eps}.
    \]
    Combining this with~\eqref{A_second_moment_mixed}, we get
    \begin{multline*}
        \int_X^{2X}|A(v)|^2|B(v)|^2\,dv \ll
        \max_{X\le v\le 2X}|B(v)|^2
        \int_X^{2X}|A(v)|^2\,dv \ll \\
        N^{-\frac{356\theta}{641}+\eps}
        M^{\frac{374}{641}+\frac{356\theta}{641}}
        X^{\frac{356}{641}+\eps}
        \cdot XN (\log N) =
        N^{1-\frac{356\theta}{641}+\eps}
        M^{\frac{374}{641}+\frac{356\theta}{641}}
        X^{\frac{997}{641}+\eps}.
    \end{multline*}
    Summing over dyadic \(X \le N^{1.55} \) we get
    \[
        \int_{N^\theta/M^{\theta-1}}^{N^{1.55}}|A(v)|^2|B(v)|^2\,dv \ll
        N^{1-\frac{356\theta}{641}+\eps}
        M^{\frac{374}{641}+\frac{356\theta}{641}}
        N^{1.55\cdot \frac{997}{641}} = M^{\frac{374}{641}+\frac{356\theta}{641}} N^{3.411 - \frac{356\theta}{641}}.
    \]

    \vspace{2ex}
    
    \textit{Range \(X>N^{1.55}\).} Here we necessarily have $T \ge N^{1.55}$.
    By a slight variation of~\cite[Lemma~3]{Robert-Sargos} we have 
    \[
        \int_0^T |A(v)|^2|B(v)|^2\,dv
        \ll N^{\eps} \frac{T}{N^{1.55}}
        \int_0^{N^{1.55}} |A(v)|^2 |B(v)|^2\,dv.
    \]
    Using the bound just proved at \(T=N^{1.55}\), we get
    \[
        \int_0^T |A(v)|^2|B(v)|^2\,dv
        \ll N^{\eps} \frac{T}{N^{1.55}}
        \Big( N^2M^2 +
            M^{\frac{374}{641}+\frac{356\theta}{641}} N^{3.411 - \frac{356\theta}{641}}
        \Big),
    \] which completes the proof.
\end{proof}

\begin{corollary} \label{energy_bound}
    For $N \ge 2$, $T > 0$, and any fixed $\theta > 1$, one has
    \[
        \int_0^T \Big| \sum_{n \asymp N} e\Big( u \frac{n^{\theta} \log n}{N^{\theta} \log N} \Big) \Big|^4 du \ll N^{4+\eps} + TN^{2.45+\eps}.
    \]
\end{corollary} This follows from Lemma~\ref{unbalanced_energy_bound}~(ii) when $M=N$.

\section{Proof for $\theta > 3$} \label{sec_main_pf_large_theta}

In this section, we prove the first part of Theorem~\ref{main_thm}. The main technical tools we will use are the counting estimates given in Propositions~\ref{count_prop_large_theta} and~\ref{count_prop_large_theta_nondiag}.

Fix a sufficiently large integer $L$ and partition the range $(3,\infty)$ into intervals
\[
    (\theta_A,\theta_B), \qquad \theta_B-\theta_A=2^{-L},
\]
where $\theta_A$ runs through the values
$3+10\cdot 2^{-L},\ 3+11\cdot 2^{-L},\ 3+12\cdot 2^{-L}, \ldots$. It is enough to prove PPC for Lebesgue-almost all $\theta\in(\theta_A,\theta_B)$ for each such interval separately. Letting $L\to\infty$ and taking a countable union then yields the desired result. Write $\eps_0:=2^{-L}$. Our goal is therefore to show that the variance
\[
    V := \int_{\theta_A}^{\theta_B} \Big|
    \frac{1}{N^2} \sum_{k\le N^{1+\eps_1}}
    \widehat f\Big(\frac{k}{N}\Big)
    \sum_{1\le m\neq n\le N}
    e\big(\alpha k(m^\theta-n^\theta)\big)
    \Big|^2 d\theta
\]
satisfies $V \ll N^{-\delta}$ with some $\delta>0$.

By symmetry between the pairs $(m,n)$ and $(n,m)$, together with the inequality $|a+b|^2\le 2|a|^2+2|b|^2$,
it is enough to consider the contribution of the range $1\le m<n\le N$. We next split the proof into two cases: 1) Near-diagonal ranges, where $m$ and $n$ are of the same order; 2) Off-diagonal ranges, where $m$ is much smaller than $n$. We treat these cases separately, although most steps are similar. We note that with our choice of parameters these ranges may overlap, but this is harmless for upper bounds.

\subsection*{Near-diagonal ranges} Here we assume that $m \asymp n$. Precisely, we write $r := n-m$, and assume that $1\le r<0.99n$. We then decompose the variables $k,n,r$ into dyadic ranges
$k\in(K,2K], n\in(Y,2Y], r\in(R,2R]$,
where the final interval in each variable may be shorter if necessary. Thus
\[
    V \ll (\log N)^3
    \sum_{\substack{K,Y,R\text{ dyadic}\\
    1\le K\le N^{1+\eps_1}\\
    1\le Y\le N\\
    1\le R\le Y}}
    V_{K,Y,R},
\]
where
\[
    V_{K,Y,R} :=
    \int_{\theta_A}^{\theta_B}
    \bigg|
    \frac{1}{N^2}
    \sum_{k\asymp K}
    \widehat f\Big(\frac{k}{N}\Big)
    \sum_{n\asymp Y}
    \sum_{\substack{r\asymp R\\ r<0.99n}}
    e\Big(\alpha k\big(n^\theta-(n-r)^\theta\big)\Big)
    \bigg|^2
    d\theta.
\]
Since there are $O((\log N)^3)$ dyadic triples $(K,Y,R)$, it is sufficient, after possibly replacing $\delta$ by a smaller number, to prove the bound $V_{K,Y,R}\ll N^{-\delta}$
uniformly for all $K,Y,R$ in the above ranges.

Opening the square, moving the sums outside the integral, and using $|\widehat f(x)|\ll 1$, we obtain
\[
    V_{K,Y,R} \ll
    \frac{1}{N^4}
    \sum_{k_1,k_2\asymp K}
    \sum_{n_1,n_2\asymp Y}
    \sum_{\substack{r_1,r_2\asymp R\\ r_i<0.99 n_i}}
    \bigg|
    \int_{\theta_A}^{\theta_B}
    e\big(E_0(\theta)\big)\,d\theta
    \bigg|,
\]
where
\[
    E_0(\theta) :=
    \alpha k_1\big(n_1^\theta-(n_1-r_1)^\theta\big) -
    \alpha k_2\big(n_2^\theta-(n_2-r_2)^\theta\big).
\]

Next, we split the interval $(\theta_A,\theta_B]$ into $H$ subintervals of equal length $\frac{\eps_0}{H}$:
\[
    (\theta_A,\theta_B] =
    \bigcup_{i=0}^{H-1}
    \Big(\theta_i,\theta_i+\frac{\eps_0}{H}\Big],
    \qquad \theta_i:=\theta_A+\frac{i\eps_0}{H}.
\]
Choose
\begin{equation} \label{def_H_first}
    H:=\big\lfloor (KY^{\theta_B-1}R)^{\frac{1}{2}+\eps_2} \big\rfloor
\end{equation}
with some small $\eps_2>0$. It is therefore enough to show that
\[
    V_{K,Y,R}^{(i)}\ll H^{-1}N^{-\delta}
\]
for every $0\le i\le H-1$, where
\[
    V_{K,Y,R}^{(i)} :=
    \frac{1}{N^4}
    \sum_{k_1,k_2\asymp K}
    \sum_{n_1,n_2\asymp Y}
    \sum_{\substack{r_1,r_2\asymp R\\ r_i<0.99n_i}}
    \bigg|
    \int_{\theta_i}^{\theta_i+\eps_0/H}
    e\big(E_0(\theta)\big)\,d\theta
    \bigg|.
\]

Put
\[
    u:=\theta-\theta_i, \qquad
    \Phi(u):=E_0(\theta_i+u),
    \qquad 0\le u\le \frac{\eps_0}{H}.
\]
Then
\[
    \Phi'(u)=E_1(\theta_i+u), \qquad
    \Phi''(u)=E_2(\theta_i+u),
\]
where for $\ell\ge 1$ we write
\begin{multline*}
    E_\ell(\theta) :=
    \alpha k_1\Big(n_1^\theta(\log n_1)^\ell-(n_1-r_1)^\theta(\log(n_1-r_1))^\ell\Big) - \\
    \alpha k_2\Big(n_2^\theta(\log n_2)^\ell-(n_2-r_2)^\theta(\log(n_2-r_2))^\ell\Big).
\end{multline*}

Since $n_i\asymp Y$, $r_i\asymp R$, and $r_i < 0.99 n_i$, we have
\[
    E_2(\theta)\ll KY^{\theta_B-1}R(\log Y)^2
\]
uniformly for $\theta_i\le \theta\le \theta_i+\frac{\eps_0}{H}$. Hence, by the mean-value theorem,
\[
    |\Phi'(u)-E_1(\theta_i)| \le
    \frac{\eps_0}{H}
    \max_{\theta_i\le \theta\le \theta_i+\eps_0/H}|E_2(\theta)| \ll
    \frac{\eps_0}{H}KY^{\theta_B-1}R(\log Y)^2.
\]
Therefore, if
\begin{equation} \label{E1_large_condition_large_theta}
    |E_1(\theta_i)| \ge
    2C\frac{\eps_0}{H}KY^{\theta_B-1}R(\log Y)^2
\end{equation}
with a sufficiently large absolute constant $C>0$, then
\[
    |\Phi'(u)|\asymp |E_1(\theta_i)|
    \qquad
    \text{for all } 0\le u\le \frac{\eps_0}{H}.
\]

In this case, substituting $u := \theta - \theta_i$ and integrating by parts, we obtain
\[
    \int_{\theta_i}^{\theta_i + \eps_0 / H} e\big( E_0 (\theta) \big) d\theta = \int_0^{\eps_0/H} e(\Phi(u))\,du=
    \bigg[\frac{e(\Phi(u))}{2\pi i\,\Phi'(u)}\bigg]_0^{\eps_0/H} -
    \int_0^{\eps_0/H}
    e(\Phi(u)) \frac{\Phi''(u)}{2\pi i\,(\Phi'(u))^2}\,du,
\]
and hence
\[
    \bigg|
    \int_0^{\eps_0/H} e(\Phi(u))\,du
    \bigg| \ll
    \frac{1}{|E_1(\theta_i)|} +
    \frac{1}{|E_1(\theta_i)|^2}
    \int_0^{\eps_0/H} |\Phi''(u)|\,du.
\]
Using again the bound for $E_2$ and the assumption~\eqref{E1_large_condition_large_theta}, we get
\[
    \int_0^{\eps_0/H} |\Phi''(u)|\,du \ll
    \frac{\eps_0}{H}KY^{\theta_B-1}R(\log Y)^2 \ll
    |E_1(\theta_i)|,
\]
and therefore
\[
    \bigg|
    \int_0^{\eps_0/H} e(\Phi(u))\,du
    \bigg| \ll
    \frac{1}{|E_1(\theta_i)|}.
\]

On the other hand, we always have the trivial bound
\[
    \bigg|
    \int_0^{\eps_0/H} e(\Phi(u))\,du
    \bigg| \le
    \frac{\eps_0}{H} \ll \frac{1}{H}.
\] Finally, note that if~\eqref{E1_large_condition_large_theta} does not hold, then we have $H^{-1} \ll |E_1 (\theta_i)|^{-1}$ by~\eqref{def_H_first}. Thus, in all cases,
\[
    \bigg|
    \int_0^{\eps_0/H} e\big(E_0(\theta_i+u)\big)\,du
    \bigg| \ll
    \min\Big(\frac{1}{H},\frac{1}{|E_1(\theta_i)|}\Big),
\]
and consequently
\[
    V_{K,Y,R}^{(i)} \ll
    \frac{1}{N^4}
    \sum_{k_1,k_2\asymp K}
    \sum_{n_1,n_2\asymp Y}
    \sum_{\substack{r_1,r_2\asymp R\\ r_i<0.99n_i}}
    \min\Big(\frac{1}{H},\frac{1}{|E_1(\theta_i)|}\Big).
\]

Next, we split the sum defining $V_{K,Y,R}^{(i)}$ into dyadic ranges according to the size of $E_1(\theta_i)$. Let
\[
    \Delta\in\Big\{1,\frac12,\frac14,\ldots,\frac{O(1)}{H}\Big\}.
\]
Then
\[
    V_{K,Y,R}^{(i)} \le
    \sum_{\substack{H^{-1}\le \Delta\le 1\\ \Delta\text{ dyadic}}}
    V_{K,Y,R}^{(i,\Delta)},
\]
where
\[
    V_{K,Y,R}^{(i,\Delta)} :=
    \frac{1}{N^4} \sum_{\substack{
    k_1,k_2\asymp K\\
    n_1,n_2\asymp Y\\
    r_1,r_2\asymp R,\ r_i<0.99n_i\\
    |E_1(\theta_i)|\asymp \Delta K Y^{\theta_i-1}R\log Y}}
    \min\Big(\frac{1}{H},\frac{1}{|E_1(\theta_i)|}\Big),
\]
and in the last block, corresponding to $\Delta\asymp H^{-1}$, the condition $|E_1(\theta_i)|\asymp \Delta K Y^{\theta_i-1}R\log Y$ is replaced by $|E_1(\theta_i)|\le \Delta K Y^{\theta_i-1}R\log Y$..

It follows that
\begin{multline*}
    V_{K,Y,R}^{(i,\Delta)} \ll
    \frac{1}{N^4}
    \#\Big\{
    k_1,k_2\asymp K,\ n_1,n_2\asymp Y,\ r_1,r_2\asymp R,\ r_i<0.99n_i : \\
    |E_1(\theta_i)|\le \Delta K Y^{\theta_i-1}R\log Y
    \Big\} \cdot
    \min\Big(\frac{1}{H},\frac{1}{\Delta K Y^{\theta_i-1}R\log Y}\Big).
\end{multline*}
Applying Proposition~\ref{count_prop_large_theta}, we obtain
\begin{multline*}
    V_{K,Y,R}^{(i,\Delta)} \ll
    \frac{1}{N^4} N^{\eps_3}
    \Big(
        K^2Y^{1.999}
        +\Delta K^2 R^2 Y^2
        +K R^{\frac{1}{2}} Y^{2.499}
        +\Delta K R^{\frac{1}{2}} Y^{\frac{9}{2}}
    \Big) \cdot \\
    \min\Big(\frac{1}{H},\frac{1}{\Delta K Y^{\theta_i-1}R\log Y}\Big)
    =: T_1+T_2+T_3+T_4,
\end{multline*}
where $T_j$ denotes the contribution of the corresponding term in the brackets.

We first estimate $T_1$ and $T_3$ using the bound $H^{-1}$ in the minimum. Since $K\ll N^{1+\eps_1}$, $Y\ll N$, and $R\le Y$, we get
\[
    T_1 \ll N^{-4+\eps_3} K^2Y^{1.999}H^{-1}
    \ll N^{-4+\eps_3+3.999+2\eps_1}H^{-1},
\]
and similarly
\[
    T_3 \ll N^{-4+\eps_3} KR^{\frac{1}{2}}Y^{2.499}H^{-1}
    \ll N^{-4+\eps_3+3.999+\eps_1} H^{-1}.
\] Thus both bounds are acceptable provided $\eps_1,\eps_3$ are chosen sufficiently small.

Next, consider $T_2$. Using the second bound in the minimum and suppressing the logarithmic factor, we obtain
\[
    T_2 \ll N^{-4+\eps_3}\,\Delta K^2R^2Y^2\cdot
    \frac{1}{\Delta K Y^{\theta_i-1}R}
    \ll N^{-4+\eps_3} KRY^{3-\theta_i}.
\] Multiplying and dividing by $H$, and using the definition of $H$~\eqref{def_H_first}, we get
\[
    T_2 \ll N^{-4+\eps_3} KRY^{3-\theta_i}
    (KY^{\theta_B-1}R)^{\frac{1}{2}+\eps_2} H^{-1}.
\] Now using $K\ll N^{1+\eps_1}$, $Y\ll N$, $R\le Y$, and $\theta_i\ge \theta_A$, we find
\begin{align*}
    T_2 &\ll N^{-4+\eps_3}
    N^{1+\eps_1}
    N \cdot N^{3-\theta_A}
    N^{(\frac12+\eps_2)(1+\eps_1+\theta_B)}
    H^{-1} \\
    &\ll N^{\frac32-\frac{\theta_A}{2}+\frac{\eps_0}{2}+2\eps_1+2\eps_2\theta_A+\eps_3} H^{-1}.
\end{align*}
Since $\theta_A \ge 3+10\eps_0$, this is
\[
    \ll N^{-5\eps_0}H^{-1}
\]
provided $\eps_1,\eps_2,\eps_3>0$ are chosen sufficiently small depending on $\eps_0$.

Finally, for $T_4$ we similarly obtain
\[
    T_4 \ll N^{-4+\eps_3}\,\Delta K R^{\frac{1}{2}} Y^{\frac{9}{2}}
    \cdot \frac{1}{\Delta K Y^{\theta_i-1}R}
    \ll N^{-4+\eps_3} R^{-\frac{1}{2}} Y^{\frac{11}{2}-\theta_i}.
\]
Multiplying and dividing by $H$, and using again the definition of $H$, gives
\[
    T_4 \ll N^{-4+\eps_3}
    R^{-\frac{1}{2}} Y^{\frac{11}{2}-\theta_i}
    (KY^{\theta_B-1}R)^{\frac{1}{2}+\eps_2} H^{-1}.
\]
Estimating trivially with $R\le Y$ and $Y\ll N$, we get
\[
    T_4 \ll
    N^{\frac32-\frac{\theta_A}{2}+\frac{\eps_0}{2}+\eps_1+2\eps_2\theta_A+\eps_3} H^{-1},
\]
which is smaller than the bound obtained for $T_2$.

Collecting the estimates for $T_1,\ldots,T_4$, we conclude that
\[
    V_{K,Y,R}^{(i,\Delta)} \ll H^{-1}N^{-\delta}
\]
for some $\delta>0$, uniformly in $K,Y,R,i,\Delta$. Summing over the dyadic values of $\Delta$ and then over $i$, yields $V_{K,Y,R}\ll N^{-\delta / 2}$, as desired.

\subsection*{Off-diagonal ranges} Here we assume that $m \asymp Y_1$, $n \asymp Y_2$ for $Y_1 \le Y_2$, and that $n-m \gg n$. A major part of the proof is identical to that in the near-diagonal range. We note that $Y_1$ now essentially plays a role of $R$ and $Y_2$ essentially plays a role of $Y$ in most steps, and the typical size of the difference $Y^{\theta-1} R \log Y$ is replaced by just $Y_2^{\theta} \log Y_2$.

Therefore, we can write
\[
    V \ll (\log N)^3 \sum_{\substack{K, Y_1, Y_2 \text{ dyadic} \\ 1 \le K \le N^{1+\eps_1} \\ 1 \le Y_2 \le N \\ 1 \le Y_1 < Y_2}} V_{K, Y_1, Y_2},
\] with
\[
    V_{K, Y_1, Y_2} := \int_{\theta_A}^{\theta_B}
    \bigg| \frac{1}{N^2} \sum_{k\asymp K}
    \widehat f\Big(\frac{k}{N}\Big)
    \sum_{n\asymp Y_2}
    \sum_{\substack{m \asymp Y_1 \\ m < 0.1n}} 
    e\Big(\alpha k\big(n^\theta-m^\theta\big)\Big)
    \bigg|^2 d\theta.
\] Next, choosing
\[
    H := \big\lfloor (KY_2^{\theta_B})^{\frac{1}{2}+\eps_2} \big\rfloor,
\] we repeat all the steps until we arrive at the inequality
\begin{multline*}
    V_{K,Y_1,Y_2}^{(i,\Delta)} := \frac{1}{N^4} \sum_{\substack{k_1, k_2 \asymp K \\ n_1, n_2 \asymp Y_2 \\ m_1, m_2 \asymp Y_1, n_i > 10 m_i \\ |E_1 (\theta_i)| \asymp \Delta KY_2^{\theta_i} \log Y_2}} \min \Big( \frac{1}{H}, \frac{1}{|E_1 (\theta_i)|} \Big) 
    \ll \\ 
    \frac{1}{N^4} \#\Big\{
    k_1,k_2\asymp K,\ m_1,m_2\asymp Y_1,\ n_1,n_2\asymp Y_2, \ n_i > 10m_i: \\
    |E_1(\theta_i)|\le \Delta K Y_2^{\theta_i} \log Y_2 \Big\} \cdot
    \min\Big(\frac{1}{H},\frac{1}{\Delta K Y_2^{\theta_i} \log Y_2 }\Big),
\end{multline*} where
\[
    E_1 (\theta) = \alpha k_1 \big( n_1^{\theta} \log n_1 - m_1^{\theta} \log m_1 \big) - \alpha k_2 \big( n_2^{\theta} \log n_2 - m_2^{\theta} \log m_2 \big). 
\] 

Next, applying Proposition~\ref{count_prop_large_theta_nondiag}, we get
\[
    V_{K,Y_1,Y_2}^{(i,\Delta)} \ll \frac{1}{N^4} (KY_2)^{\eps_3} \Big( K^2 Y_2^{1.999} + \Delta K^2 Y_2^4 + KY_2^{2.999} + \Delta K Y_2^5  \Big) \min \Big( \frac{1}{H}, \frac{1}{\Delta K Y_2^{\theta_i} \log Y_2} \Big).
\] The rest of the computation is the same as in the near-diagonal range after taking $R = Y_2 = Y$.

We remark that the restriction $\theta < 7$ in Proposition~\ref{count_prop_large_theta_nondiag} is rather technical (note that the result for $\theta > 7$ follows from~\cite{Technau_Yesha}), and the upper threshold can easily be increased to a larger value, for example to $\theta = 8$. It could perhaps be removed completely by a more careful analysis or by splitting some parameters into additional ranges, but for the sake of simplicity we do not pursue this here.

\section{Proof for $0<\theta<\frac{3}{5}$} \label{sec_main_pf_small_theta}

In this section, we prove the second part of Theorem~\ref{main_thm}, assuming Propositions~\ref{count_small_theta_diag} and~\ref{count_small_theta_nondiag}.

\subsection*{Setup}

The crucial input in the argument for small $\theta$ is the following statement, which is essentially a consequence of a triple Poisson summation.

\begin{proposition} \label{prop_1}
	Let $\alpha, \theta, \varepsilon$ be real numbers such that $\alpha > 0$, $\varepsilon > 0$, and $3\varepsilon < \theta < 1 - 3\varepsilon$. For $K, Y_1, Y_2 > 0$ define the subset $\mathcal{N}_{\theta}(K,Y_1,Y_2)$ by
	\begin{gather*}
		\mathcal{N}_{\theta}(K,Y_1,Y_2) :=
		\bigl\{ (m_1,m_2): \\
		\alpha \theta K (2Y_1)^{\theta-1} \le m_1 < 2 \alpha \theta K Y_1^{\theta-1}, \\
		\alpha \theta K (2Y_2)^{\theta-1} \le m_2 < 2 \alpha \theta K Y_2^{\theta-1}, \\
		m_1 < m_2
		\bigr\}.
	\end{gather*}
	Then we have
	\begin{equation} \label{main_transform}
		\frac{2}{N^2} \sum_{k \le N^{1+\varepsilon}} \hat{f} \bigl( \tfrac{k}{N} \bigr) \Big| \sum_{1 \le y \le N} e\bigl( \alpha k y^{\theta} \bigr) \Big|^2 =
		f(0) + o(1) + O\bigl(S(N)\bigr),
	\end{equation}
	where
	\begin{equation} \label{new_expsum}
		S(N) :=
		\frac{1}{N^2}
		\sum_{\substack{K \le N^{1+\varepsilon} \\ \text{dyadic}}}
		\sum_{\substack{Y_1, Y_2 \le N \\ \text{dyadic}}}
		K^{\frac{1}{1-\theta}} \bigl| S(K,Y_1,Y_2) \bigr|,
	\end{equation}
	\begin{equation} \label{dyadic_sum}
		S(K,Y_1,Y_2) :=
		\sum_{(m_1,m_2) \in \mathcal{N}_{\theta}(K,Y_1,Y_2)}
		\Big( \frac{1}{m_1 m_2} \Big)^{\frac{2-\theta}{2-2\theta}}
		\sum_{L_1 \le \ell < L_2}
		\frac{c_1}{\sqrt{\eta}}
		\Big(\frac{\ell}{\eta}\Big)^{\frac{1}{2\theta} - 1}
		e\bigl( c_2\ell^{\frac{1}{\theta}} \eta^{1 - \frac{1}{\theta}} \bigr),
	\end{equation}
	\begin{gather*}
		L_1 = \frac{c_3 \eta}{1-\theta} K_1^{\frac{\theta}{1-\theta}}, \qquad
		L_2 = \frac{c_3 \eta}{1-\theta} K_2^{\frac{\theta}{1-\theta}}, \qquad
		c_3 = (\alpha \theta)^{\frac{1}{1-\theta}} \Big( \frac{1}{\theta} - 1 \Big), \\
		c_1 = \frac{1 - \theta}{\sqrt{c_3 \theta}} \Big( \frac{1-\theta}{c_3} \Big)^{\frac{1}{2\theta} - 1}, \qquad
		c_2 = -\theta \Big( \frac{1-\theta}{c_3} \Big)^{\frac{1}{\theta} - 1},
	\end{gather*}
	\begin{gather}
		K_1 := \max \Big( K, \frac{m_1}{2\alpha \theta Y_1^{\theta-1}},  \frac{m_2}{2\alpha \theta Y_2^{\theta-1}} \Big), \label{K_1_small} \\
		K_2 := \min \Big( 2K, \frac{m_1}{\alpha \theta (2Y_1)^{\theta-1}}, \frac{m_2}{\alpha \theta (2Y_2)^{\theta-1}} \Big), \label{K_2_small}
	\end{gather}
	with
	\[
		\eta(m_1,m_2) := m_1^{-\frac{\theta}{1-\theta}} - m_2^{-\frac{\theta}{1-\theta}}.
	\]
	The implied constant in the error term in~\eqref{main_transform} depends on $f$, $\alpha$, $\theta$, and $\varepsilon$.
\end{proposition}

\begin{proof}
    This is a slightly adjusted version of~\cite[Proposition~1.3]{Shubin_Radziwill}. Since the main part of the proof is identical, we only explain the final modification.

    The main difference from~\cite[Proposition~1.3]{Shubin_Radziwill} is that here we wish to avoid taking the maximum over $K\le \widetilde K\le 2K$. To achieve this, after Poisson summation in $m_1,m_2$, we do not apply partial summation in $k$, but instead interchange the order of summation over $k$ and over $m_1,m_2$. Then the term denoted by $E_{1,1}(K)$ in~\cite{Shubin_Radziwill} becomes
    \begin{multline*}
        E_{1,1}(K)
        =
        \sum_{\substack{Y_1,Y_2 \le N\\ \text{dyadic}}}
        \sum_{\substack{m_1\neq m_2\\
        \alpha\theta K(2Y_i)^{\theta-1}\le m_i<2\alpha\theta K Y_i^{\theta-1}}}
        \Big(\frac{1}{m_1m_2}\Big)^{\frac{2-\theta}{2-2\theta}} \cdot \\
        \sum_{K_1\le k<K_2}
        \widehat f\Big(\frac{k}{N}\Big)
        \Big(\frac{c_4}{\sqrt{k}}\Big)^2
        k^{\frac{2-\theta}{1-\theta}}
        e\big(c_3 k^{\frac{1}{1-\theta}}\eta\big),
    \end{multline*}
    where
    \[
        K_1 = \max \Big( K, \frac{m_1}{2\alpha \theta Y_1^{\theta-1}}, \frac{m_2}{2\alpha \theta Y_2^{\theta-1}} \Big), \qquad
        K_2 = \min \Big( 2K, \frac{m_1}{\alpha \theta (2Y_1)^{\theta-1}}, \frac{m_2}{\alpha \theta (2Y_2)^{\theta-1}} \Big).
    \]
    Applying~\cite[(8.47)]{IK} to the sum over $k$ in place of~\cite[Theorem~8.16]{IK}, one obtains~\eqref{dyadic_sum}, and this completes the proof.
\end{proof}

Proposition~\ref{prop_1} shows that it is sufficient to prove that for almost all $\theta \in (0, \frac{3}{5})$ one has
\[
    S(N)=o(1)
    \qquad (N\to\infty).
\] Since there are only $O((\log N)^3)$ dyadic triples $(K,Y_1,Y_2)$, it is enough to show that, for all dyadic values of $K,Y_1,Y_2$ such that $Y_1,Y_2\le N$ and $K\le N^{1+\eps_1}$, and for almost all $\theta\in(\theta_A,\theta_B)\subset (0,\frac35)$, one has
\[
    \frac{K^{\frac{1}{1-\theta}}}{N^2}\,|S(K,Y_1,Y_2)| \le N^{-\delta}
\] with some absolute $\delta>0$. 

We now write
\[
    m_1=:m, \qquad m_2=:m+r.
\]
Next, we again decompose the sum \(S(K,Y_1,Y_2)\) into a near-diagonal part and an off-diagonal part. Namely,
\[
    |S(K,Y_1,Y_2)| \le
    \Id_{Y_1 \le 10Y_2}
    \sum_{\substack{R \text{ dyadic}\\ R\ll K Y_1^{\theta-1}}}
    \big| S_1(K,Y_1,Y_2,R) \big| +
    \Id_{Y_1>10Y_2} \big| S_2(K,Y_1,Y_2) \big|,
\]
where
\[
    S_1(K,Y_1,Y_2,R) :=
    \sum_{(m,r)\in \mathcal M_\theta(K,Y_1,Y_2,R)}
    \Big(\frac{1}{m(m+r)}\Big)^{\frac{2-\theta}{2-2\theta}}
    \sum_{L_1\le \ell <L_2}
    \frac{c_1}{\sqrt{\eta}}
    \Big(\frac{\ell}{\eta}\Big)^{\frac{1}{2\theta}-1}
    e\big(c_2\ell^{\frac1\theta}\eta^{1-\frac1\theta}\big),
\]
with
\begin{multline*}
    \mathcal M_\theta(K,Y_1,Y_2,R) :=
    \Big\{ (m,r): \ 
    \alpha\theta K(2Y_1)^{\theta-1}\le m<2\alpha\theta K Y_1^{\theta-1}, \\
    \alpha\theta K(2Y_2)^{\theta-1} \le m+r < 2\alpha\theta K Y_2^{\theta-1},
    \ R<r\le 2R \Big\},
\end{multline*}
and
\[
    S_2(K,Y_1,Y_2):=S(K,Y_1,Y_2).
\]
Thus, in the case where $m_1$ and $m_2$ are close to each other, we further localize the difference $r=m_2-m_1$ dyadically. When $m_1$ and $m_2$ are far apart, as certainly happens if $Y_1>10Y_2$, we work with the full sum. The reason for this decomposition is that the phase and its derivatives have different natural sizes in the two regimes.

We are therefore reduced to proving the variance bounds
\begin{equation} \label{original_expressions}
    \int_{\theta_A}^{\theta_B}
    \Big| \frac{K^{\frac{1}{1-\theta}}}{N^2} S_1(K,Y_1,Y_2,R) \Big|^2 d\theta
    \ll N^{-\delta}, \qquad
    \int_{\theta_A}^{\theta_B}
    \Big| \frac{K^{\frac{1}{1-\theta}}}{N^2} S_2(K,Y_1,Y_2) \Big|^2 d\theta
    \ll N^{-\delta},
\end{equation}
uniformly in $K,Y_1,Y_2,R$.

As in the previous section, we split these integrals into shorter ones:
\[
    I_1^{(i)} :=
    \int_{\theta_i}^{\theta_i+ \eps_0 / H_1}
    \Big| \frac{K^{\frac{1}{1-\theta}}}{N^2} S_1(K,Y_1,Y_2,R) \Big|^2 d\theta,
    \qquad I_2^{(i)} :=
    \int_{\theta_i}^{\theta_i+ \eps_0 / H_2}
    \Big| \frac{K^{\frac{1}{1-\theta}}}{N^2} S_2(K,Y_1,Y_2) \Big|^2 d\theta,
\]
where
\[
    H_1:= \big\lfloor (RY_1)^{\frac{1}{2}+\eps_2} \big\rfloor,
    \qquad H_2:= \big\lfloor (K Y_1^{\theta_B})^{\frac{1}{2}+\eps_2} \big\rfloor.
\] It is enough to prove
\[
    I_1^{(i)}\ll H_1^{-1}N^{-\delta},
    \qquad I_2^{(i)}\ll H_2^{-1}N^{-\delta}.
\]

\subsection*{Simplifying the sums}

We now make several technical reductions that bring~\eqref{original_expressions} into a more convenient form. We will show that it is enough to prove
\begin{gather*}
    \frac{Y_1^{3-2\theta_i}}{N^4R}
    \int_{\theta_i}^{\theta_i + \eps_0 / H_1}
    \bigg| \sum_{(m,r)\in \mathcal M_{\theta_i}(K,Y_1,Y_2,R)}
    \sum_{L_1(\theta_i)\le \ell <L_2(\theta_i)}
    e\big(c_2\ell^{\frac1\theta}\eta^{1-\frac1\theta}\big)
    \bigg|^2 d\theta
    \ll H_1^{-1}N^{-\delta}, \\
    \frac{Y_1^{2-2\theta_i}Y_2^{2-\theta_i}}{N^4K}
    \int_{\theta_i}^{\theta_i + \eps_0 / H_2}
    \bigg| \sum_{(m_1,m_2)\in \mathcal N_{\theta_i}(K,Y_1,Y_2)}
    \sum_{L_1(\theta_i)\le \ell <L_2(\theta_i)}
    e\big(c_2\ell^{\frac1\theta}\eta^{1-\frac1\theta}\big)
    \bigg|^2 d\theta
    \ll H_2^{-1}N^{-\delta}.
\end{gather*}

These new integrals are obtained from the original ones $I_1^{(i)}$ and $I_2^{(i)}$ by replacing the endpoints $L_j(\theta)$ by $L_j(\theta_i)$, replacing the sets $\mathcal M_\theta,\mathcal N_\theta$ by $\mathcal M_{\theta_i},\mathcal N_{\theta_i}$, and then applying partial summation to the sums over $\ell,m,r$ and over $\ell,m_1,m_2$, respectively.

\vspace{2ex}
\textit{Step 1. Replace the endpoints in the sum over \(\ell\).}

Recall that
\[
    L_1(\theta)=\frac{c_3\eta}{1-\theta} K_1^{\frac{\theta}{1-\theta}}.
\]
By the mean-value theorem,
\[
    L_1(\theta)-L_1(\theta_i)
    =
    L_1'(\theta_0)\cdot \frac{\eps_0}{H_j}
\]
for some \(\theta_0\in[\theta_i,\theta]\), where \(j=1\) or \(j=2\), depending on the case.

We first consider \(I_1^{(i)}\). Here
\[
    \eta(m,r)=m^{-\frac{\theta}{1-\theta}}-(m+r)^{-\frac{\theta}{1-\theta}}
    \asymp r m^{-\frac{1}{1-\theta}}
    \asymp R(KY_1^{\theta-1})^{-\frac{1}{1-\theta}},
\]
and similarly
\[
    \eta_\theta'(m,r) \asymp
    R(KY_1^{\theta-1})^{-\frac{1}{1-\theta}}\log(KY_1^{\theta-1}).
\]
Therefore
\[
    L_1'(\theta) \ll
    R(KY_1^{\theta-1})^{-\frac{1}{1-\theta}}(\log K)K^{\frac{\theta}{1-\theta}}
    \ll \frac{RY_1}{K}\log K,
\]
and hence
\[
    L_1(\theta)-L_1(\theta_i) \ll_{\eps_0}
    \frac{RY_1}{KH_1}\log K.
\]
Thus, replacing \(L_1(\theta)\) by \(L_1(\theta_i)\) changes the sum over \(\ell\) by at most
\[
    \max\Big(1,\frac{RY_1}{KH_1}\log K\Big)
\]
terms. Estimating these extra terms trivially and using the bound $|a+b|^2 \le 2|a|^2 + 2|b|^2$, we get the following bound for the error term:
\begin{multline*}
    \frac{Y_1^{3-2\theta_i}}{N^4 R} \int_{\theta_i}^{\theta_i + \eps_0 / H_1} \Big| \sum_{m,r} \max \Big( 1, \frac{RY_1}{KH_1} \log K \Big) \Big|^2 d\theta \ll \\
    \frac{1}{H_1}
    \Big( \frac{K^2 RY_1}{N^4} +
    \frac{R^2Y_1^2}{N^4} \Big)
    \ll \frac{1}{H_1}
    \Big( N^{3+3\eps_1+\frac35-4} +
    N^{2+2\eps_1+2\cdot\frac35-4} \Big),
\end{multline*} which is clearly sufficient.

For \(I_2^{(i)}\), we have
\[
    \eta(m_1,m_2) =
    m_1^{-\frac{\theta}{1-\theta}}-m_2^{-\frac{\theta}{1-\theta}}
    \asymp m_1^{-\frac{\theta}{1-\theta}}
    \asymp (KY_1^{\theta-1})^{-\frac{\theta}{1-\theta}},
\]
and therefore
\[
    L_1'(\theta) \ll
    (KY_1^{\theta-1})^{-\frac{\theta}{1-\theta}}(\log K)K^{\frac{\theta}{1-\theta}} \ll
    Y_1^\theta \log K.
\]
Thus
\[
    L_1(\theta)-L_1(\theta_i) \ll_{\eps_0}
    \frac{Y_1^\theta \log K}{H_2},
\]
and replacing \(L_1(\theta)\) by \(L_1(\theta_i)\) changes the sum by at most
\[
    \max\Big(1,\frac{Y_1^\theta \log K}{H_2}\Big)
\]
extra terms. Estimating these trivially, we obtain an error
\[
    \ll \frac{1}{H_2}
    \Big( \frac{K^3Y_2^{\theta_B}}{N^4} +
    \frac{K^2Y_1^{\theta_B}Y_2^{\theta_B}}{N^4}
    \Big) \ll \frac{1}{H_2}
    \Big( N^{3+3\eps_1+\frac35-4}
    + N^{2+2\eps_1+2\cdot\frac35-4} \Big),
\] which is again sufficient.

The upper limit \(L_2(\theta)\) is replaced by \(L_2(\theta_i)\) in exactly the same way.

\vspace{2ex}
\textit{Step 2. Replace the sets \(\mathcal M_\theta,\mathcal N_\theta\) by \(\mathcal M_{\theta_i},\mathcal N_{\theta_i}\).}

The argument is similar to the previous one. Here the relevant endpoint is
\[
    M(\theta):=2\alpha\theta K Y_1^{\theta-1},
\] so
\[
    M'(\theta)\ll K Y_1^{\theta-1}\log Y_1,
    \qquad
    M(\theta)-M(\theta_i)\ll_{\eps_0}\frac{K Y_1^{\theta-1}\log Y_1}{H_j}.
\]

In the near-diagonal case this leads to an error
\[
    \ll \frac{1}{H_1}
    \Big( \frac{Y_1^{5-2\theta_A}R^3}{N^4K^2}
    + \frac{Y_1^2R^2}{N^4} \Big)
    \ll \frac{1}{H_1} \Big(
    N^{1+\eps_1+2\eps_0+2+\frac35-4}
    + N^{2+2\eps_1+2\cdot\frac35-4} \Big),
\] which is sufficient.

In the off-diagonal case the largest contribution comes from the (shorter) \(m_1\)-sum, and similarly one gets
\[
    \ll \frac{1}{H_2}
    \Big( \frac{K Y_2^{\theta_B}Y_1^2}{N^4}
    + \frac{K^2Y_2^{\theta_B}Y_1^{\theta_B}}{N^4} \Big)
    \ll \frac{1}{H_2} \Big( N^{1+\eps_1+\frac35+2-4}
    + N^{2+2\eps_1+2\cdot\frac35-4} \Big),
\]
which is also sufficient. The same argument applies when replacing the endpoints for the \(m_2\)-sum.

\vspace{2ex}
\textit{Step 3. Replace \(\theta\) by \(\theta_i\) in the smooth coefficients.}

The coefficients in front of the exponential sums $S(K, Y_1, Y_2, R)$ and $S(K, Y_1, Y_2)$ are smooth functions of \(\theta\), built out of powers of \(\ell,m,m+r,\eta\), and similarly in the off-diagonal case out of \(\ell,m_1,m_2,\eta\). By Taylor expansion at \(\theta=\theta_i\), each such coefficient changes by a relative factor
\[
    \ll \frac{\log K}{H_j},
\]
and therefore every such replacement contributes an error of the same order as in the previous two steps. Consequently, we may replace all occurrences of \(\theta\) in the smooth weights by \(\theta_i\).

\vspace{2ex}
\textit{Step 4. Partial summation.}

Applying partial summation in \(\ell,m,r\) for \(I_1^{(i)}\) and in \(\ell,m_1,m_2\) for \(I_2^{(i)}\), we obtain
\begin{gather*}
    I_1^{(i)} \ll
    \frac{Y_1^{3-2\theta_i}}{N^4R}
    \int_{\theta_i}^{\theta_i+ \eps_0 / H_1}
    \bigg| \sum_{(m,r)\in \mathcal M_{\theta_i}(K,Y_1,Y_2,R)}
    \sum_{L_1(\theta_i)\le \ell <L_2(\theta_i)}
    e\big(c_2\ell^{\frac1\theta}\eta^{1-\frac1\theta}\big) \bigg|^2 d\theta, \\
    I_2^{(i)}
    \ll \frac{Y_1^{2-2\theta_i}Y_2^{2-\theta_i}}{N^4K} \int_{\theta_i}^{\theta_i+ \eps_0 / H_2}
    \bigg| \sum_{(m_1,m_2)\in \mathcal N_{\theta_i}(K,Y_1,Y_2)}
    \sum_{L_1(\theta_i)\le \ell <L_2(\theta_i)}
    e\big(c_2\ell^{\frac1\theta}\eta^{1-\frac1\theta}\big)
    \bigg|^2 d\theta.
\end{gather*}

\subsection*{Integral evaluation}

Opening the squares and moving the sums outside the integrals, we get
\begin{gather} \label{Int_1_eval}
    I_1^{(i)} \ll \frac{Y_1^{3-2\theta_i}}{N^4R}
    \sum_{m_1,m_2\asymp K Y_1^{\theta_i-1}}
    \sum_{\substack{r_1,r_2\asymp R\\ m_j+r_j\asymp K Y_2^{\theta_i-1}}}
    \sum_{\ell_1,\ell_2\asymp RY_1K^{-1}}
    \bigg| \int_{\theta_i}^{\theta_i+ \eps_0 / H_1}
    e\big( c_2\ell_1^{\frac1\theta}\eta_1^{1-\frac1\theta} - c_2\ell_2^{\frac1\theta}\eta_2^{1-\frac1\theta}
    \big) d\theta \bigg|,
    \\ \label{Int_2_eval}
    I_2^{(i)} \ll
    \frac{Y_1^{2-2\theta_i}Y_2^{2-\theta_i}}{N^4K}
    \sum_{\substack{m_1,m_3\asymp K Y_1^{\theta_i-1}\\ m_2,m_4\asymp K Y_2^{\theta_i-1}}}
    \sum_{\ell_1,\ell_2\asymp Y_1^{\theta_i}}
    \bigg| \int_{\theta_i}^{\theta_i+ \eps_0 / H_2}
    e\big(c_2\ell_1^{\frac1\theta}\eta_1^{1-\frac1\theta} - c_2\ell_2^{\frac1\theta}\eta_2^{1-\frac1\theta}
    \big) d\theta \bigg|,
\end{gather}
where
\[
    \eta_1 = m_1^{-\frac{\theta}{1-\theta}}-(m_1+r_1)^{-\frac{\theta}{1-\theta}},
    \qquad \eta_2 =
    m_2^{-\frac{\theta}{1-\theta}}-(m_2+r_2)^{-\frac{\theta}{1-\theta}}
\]
in the near-diagonal case, and
\[
    \eta_1 = m_1^{-\frac{\theta}{1-\theta}}-m_2^{-\frac{\theta}{1-\theta}}, \qquad
    \eta_2 = m_3^{-\frac{\theta}{1-\theta}}-m_4^{-\frac{\theta}{1-\theta}}
\]
in the off-diagonal case. Since we only need upper bounds for the absolute values of these integrals, the implied constants in the $\asymp$-conditions under the sums may be chosen sufficiently small or large so that the resulting summation ranges are independent at the endpoints.

Define
\[
    E_0(\theta)
    := c_2\ell_1^{\frac1\theta}\eta_1^{1-\frac1\theta} -
    c_2\ell_2^{\frac1\theta}\eta_2^{1-\frac1\theta}.
\]
Differentiating, we find
\[
    E_1(\theta) := \frac{d}{d\theta} E_0 (\theta) = F_1(\ell_1,\eta_1,\theta)\ell_1^{\frac1\theta}\eta_1^{1-\frac1\theta} - F_1(\ell_2,\eta_2,\theta)\ell_2^{\frac1\theta}\eta_2^{1-\frac1\theta},
\]
where
\[
    F_1(\ell,\eta,\theta)
    := D(\theta)
    -\frac{1}{\theta^2}\log \ell
    +\frac{1}{\theta^2}\log \eta
    +\Big(1-\frac{1}{\theta}\Big)\frac{\eta'(\theta)}{\eta(\theta)},
\]
and \(D(\theta)\) depends only on \(\theta\).

In the near-diagonal case we write $m+r=\beta m$, $\beta > 1$. Then
\[
    \eta(\theta) =
    \big(1-\beta^{-\frac{\theta}{1-\theta}}\big)m^{-\frac{\theta}{1-\theta}},
\] so
\[
    \log \eta(\theta) =
    \log\big(1-\beta^{-\frac{\theta}{1-\theta}}\big)
    - \frac{\theta}{1-\theta}\log m.
\]
A direct computation gives
\[
    \frac{\eta'(\theta)}{\eta(\theta)}
    = -\frac{1}{(1-\theta)^2}\log m
    + \frac{1}{(1-\theta)^2}
    \frac{\beta^{-\frac{\theta}{1-\theta}}\log \beta}
    {1-\beta^{-\frac{\theta}{1-\theta}}}.
\]
Therefore
\[
    F_1(\ell,\eta,\theta)
    = D(\theta)
    -\frac{1}{\theta^2}\log \ell
    +\frac{1}{\theta^2}\log\big(1-\beta^{-\frac{\theta}{1-\theta}}\big)
    -\frac{1}{\theta(1-\theta)}
    \frac{\beta^{-\frac{\theta}{1-\theta}}\log\beta}
    {1-\beta^{-\frac{\theta}{1-\theta}}},
\]
and hence
\begin{equation} \label{F_typical_size}
    F_1(\ell,\eta,\theta)\asymp -\log \ell
\end{equation}
for all sufficiently large \(\ell\). The same conclusion holds in the off-diagonal regime as well. For the higher derivatives, we do not need explicit formulas; it is enough to note that for each fixed \(k\ge 2\), the corresponding factor \(F_k(\ell,\eta,\theta)\) is
$\ll (\log(KY_1Y_2R))^k$.

We now estimate the short integrals exactly as in Section~\ref{sec_main_pf_large_theta}. Let
\[
    \Phi_j(u):=E_0(\theta_i+u), \qquad
    0\le u\le \frac{\eps_0}{H_j},
    \qquad j=1,2.
\]
Then
\[
    \Phi_j'(u)=E_1(\theta_i+u), \qquad
    \Phi_j''(u)=E_2(\theta_i+u).
\]
The natural sizes of \(E_1\) are
\[
    G_1:=RY_1\log(RY_1K^{-1}), \qquad
    G_2:=K Y_1^{\theta_i}\log Y_1.
\]
Moreover,
\[
    E_2(\theta)\ll G_j(\log N)
\]
uniformly in \([\theta_i,\theta_i+\eps_0/H_j]\), and by the definition of \(H_j\) this implies
\[
    |\Phi_j'(u)-E_1(\theta_i)| \ll
    \frac{\eps_0}{H_j}G_j(\log N).
\]
Thus, exactly as in the previous section, we obtain
\[
    \bigg| \int_0^{\eps_0/H_j} e(\Phi_j(u))\,du \bigg| \ll
    \min\Big(\frac{1}{H_j},\frac{1}{|E_1(\theta_i)|}\Big).
\]
Consequently, from~\eqref{Int_1_eval} and~\eqref{Int_2_eval},
\[
    I_j^{(i)} \ll W_j
    \sum_{\ell_1,\ell_2}
    \sum_{\eta_1,\eta_2 \in \mathcal{N}}
    \min\Big(\frac{1}{H_j},\frac{1}{|E_1(\theta_i)|}\Big),
\]
where
\[
    W_1=\frac{Y_1^{3-2\theta_i}}{N^4R}, \qquad
    W_2=\frac{Y_1^{2-2\theta_i}Y_2^{2-\theta_i}}{N^4K},
\] and $\mathcal{N}$ is the set of the form
\[
    \mathcal{N} = \Big\{ m^{-\frac{\theta}{1-\theta}} - (m+r)^{-\frac{\theta}{1-\theta}}: \ m \asymp KY_2^{\theta_i-1}, r \asymp R \Big\}
\] or 
\[
    \mathcal{N} = \Big\{ m_1^{-\frac{\theta}{1-\theta}} - m_2^{-\frac{\theta}{1-\theta}}: \ m_1 \asymp KY_1^{\theta_i-1}, m_2 \asymp KY_2^{\theta_i-1} \Big\}.
\]

We now split the sums dyadically according to the size of $|E_1(\theta_i)|$:
\[
    I_j^{(i)} \le
    \sum_{\substack{H_j^{-1}\le \Delta\le 1\\ \Delta\text{ dyadic}}}
    I_j^{(i,\Delta)},
\]
where
\[
    I_j^{(i,\Delta)} := W_j
    \sum_{\substack{\ell_1,\ell_2\\ \eta_1,\eta_2\\ |E_1(\theta_i)|\asymp \Delta G_j}} \min\Big(\frac{1}{H_j},\frac{1}{|E_1(\theta_i)|}\Big),
\]
and in the last block, corresponding to \(\Delta\asymp H_j^{-1}\), the condition \(\asymp\) is replaced by \(\le\).

For the near-diagonal contribution we therefore get
\begin{multline*}
    I_1^{(i,\Delta)} \ll
    \frac{Y_1^{3-2\theta_i}}{N^4R}
    \#\Big\{
    \ell_1,\ell_2\asymp RY_1K^{-1},\ m_1,m_2\asymp K Y_1^{\theta_i-1},\ r_1,r_2\asymp R: \\
    |E_1(\theta_i)|\le \Delta RY_1\log(RY_1K^{-1})
    \Big\} \cdot
    \min\Big(\frac{1}{H_1},\frac{1}{\Delta RY_1\log(RY_1K^{-1})}\Big).
\end{multline*}
Note that \(r_j\le 100m_j\) for \(j=1,2\) in the near-diagonal regime due to the restrictions on $Y_1, Y_2, R$. Applying Proposition~\ref{count_small_theta_diag}, we obtain
\begin{multline*}
    I_1^{(i,\Delta)}
    \ll \frac{Y_1^{3-2\theta_i}}{N^4R}
    (K Y_1 R)^{\eps_3} \Big(
        K^2Y_1^{2\theta_i-2}R^2
        + K^{\frac43}Y_1^{\frac{8\theta_i}{3}-\frac43}R^2
        + \Delta K^2Y_1^{4\theta_i-2}R^2
    \Big) \cdot \\
    \min\Big(\frac{1}{H_1},\frac{1}{\Delta RY_1\log(RY_1K^{-1})}\Big) =: T_1+T_2+T_3.
\end{multline*}

For the first term, we use the bound \(H_1^{-1}\) in the minimum:
\[
    T_1 \ll
    \frac{K^2Y_1R}{N^4}(K Y_1 R)^{\eps_3}\frac{1}{H_1} \ll \frac{K^3Y_1^{\theta_i}}{N^4}(K^2Y_1^{\theta_i})^{\eps_3}\frac{1}{H_1}
    \ll N^{3+3\eps_1+\frac35+3\eps_3-4}\frac{1}{H_1},
\]
which is sufficient.

Similarly,
\[
    T_2 \ll
    \frac{K^{\frac{4}{3}}Y_1^{\frac53+\frac{2\theta_i}{3}}R}{N^4}(K Y_1 R)^{\eps_3}\frac{1}{H_1} \ll \frac{K^{\frac{7}{3}}Y_1^{\frac23+\frac{5\theta_i}{3}}}{N^4}(K^2Y_1^{\theta_i})^{\eps_3}\frac{1}{H_1} \ll    N^{\frac73+\frac73\eps_1+\frac23+\frac{5\theta_B}{3}+3\eps_3-4}\frac{1}{H_1},
\]
which is acceptable provided $\theta_B<\frac{3}{5}-2\eps_1-2\eps_3 - \delta$.

For the third term, we use the second bound in the minimum:
\[
    T_3 \ll
    \frac{Y_1^{3-2\theta_i}}{N^4R}
    (K Y_1 R)^{\eps_3}
    \cdot \Delta K^2Y_1^{4\theta_i-2}R^2
    \cdot \frac{1}{\Delta RY_1\log(RY_1K^{-1})}
    \ll \frac{K^2Y_1^{2\theta_i}}{N^4}(K Y_1 R)^{\eps_3}.
\]
Multiplying by \(H_1\), we obtain
\[
    H_1 T_3 \ll (RY_1)^{\frac{1}{2}+\eps_2}
    \frac{K^2 Y_1^{2\theta_i}}{N^4}
    (K Y_1 R)^{\eps_3} \ll
    N^{\frac52+\frac52\eps_1+2\eps_2+\frac{5\theta_B}{2}-4+3\eps_3},
\]
which is \(O(N^{-\delta})\) provided $   \theta_B<\frac{3}{5}-2(\eps_1+\eps_2+\eps_3) - \delta$.

For the off-diagonal contribution, we similarly obtain
\begin{multline*}
    I_2^{(i,\Delta)} \ll
    \frac{Y_1^{2-2\theta_i}Y_2^{2-\theta_i}}{N^4K}
    \#\Big\{
    \ell_1,\ell_2\asymp Y_1^{\theta_i},\
    m_1,m_3\asymp K Y_1^{\theta_i-1},\
    m_2,m_4\asymp K Y_2^{\theta_i-1}: \\
    |E_1(\theta_i)|\le \Delta K Y_1^{\theta_i}\log(KY_1Y_2)
    \Big\} \cdot
    \min\Big(\frac{1}{H_2},\frac{1}{\Delta K Y_1^{\theta_i}\log(KY_1Y_2)}\Big).
\end{multline*}
Applying Proposition~\ref{count_small_theta_nondiag}, we get
\begin{multline*}
    I_2^{(i,\Delta)} \ll
    \frac{Y_1^{2-2\theta_i}Y_2^{2-\theta_i}}{N^4K}
    (K Y_1 Y_2)^{\eps_3}
    \Big( K^4Y_1^{2\theta_i-2}Y_2^{2\theta_i-2}
        + K^{\frac{10}{3}}Y_1^{\frac{8\theta_i}{3}-\frac43}Y_2^{2\theta_i-2} +
        \\ \Delta K^4Y_1^{4\theta_i-2}Y_2^{2\theta_i-2} \Big) \cdot
    \min\Big(\frac{1}{H_2},\frac{1}{\Delta K Y_1^{\theta_i}\log(KY_1Y_2)}\Big)
    =: T_4+T_5+T_6.
\end{multline*}

The first term is
\[
    T_4 \ll \frac{K^3Y_2^{\theta_i}}{N^4}(K Y_1 Y_2)^{\eps_3}\frac{1}{H_2}
    \ll N^{3+3\eps_1+\frac35+3\eps_3-4}\frac{1}{H_2},
\]
which is sufficient.

Similarly,
\[
    T_5 \ll \frac{K^{7/3}Y_1^{\frac23+\frac{2\theta_i}{3}}Y_2^{\theta_i}}{N^4}
    (K Y_1 Y_2)^{\eps_3}\frac{1}{H_2} \ll
    N^{\frac73+\frac73\eps_1+\frac23+\frac{5\theta_B}{3}+3\eps_3-4}\frac{1}{H_2},
\]
which gives the same condition as that for \(T_2\) (with $H_1$ replaced by $H_2$).

Finally,
\[
    T_6 \ll
    \frac{Y_1^{2-2\theta_i}Y_2^{2-\theta_i}}{N^4K}
    (K Y_1 Y_2)^{\eps_3} \cdot
    \Delta K^4Y_1^{4\theta_i-2}Y_2^{2\theta_i-2}
    \cdot \frac{1}{\Delta K Y_1^{\theta_i}\log(KY_1Y_2)}
    \ll \frac{K^2Y_1^{\theta_i}Y_2^{\theta_i}}{N^4}(K Y_1 Y_2)^{\eps_3}.
\]
Multiplying by \(H_2\), we get
\[
    H_2 T_6 \ll (K Y_1^{\theta_i})^{\frac{1}{2}+\eps_2}
    \frac{K^2Y_1^{\theta_i}Y_2^{\theta_i}}{N^4}
    (K Y_1 Y_2)^{\eps_3} \ll
    N^{\frac52+\frac52\eps_1+2\eps_2+\frac{5\theta_B}{2}-4+3\eps_3},
\]
which is exactly the same bound as for \(T_3\).

Thus, all contributions are acceptable, provided \(\theta_B<\frac35\) and the auxiliary parameters \(\eps_1,\eps_2,\eps_3\) are chosen sufficiently small with respect to $\frac{3}{5} - \theta_B$. This completes the proof.

\section{Counting for large $\theta$} \label{sec_counting_large}

In this section, we prove two counting estimates used in the proof of Theorem~\ref{main_thm} in the case $\theta>3$.

\begin{proposition} \label{count_prop_large_theta}
    Let $K,Y,R\ge 1$ be real numbers such that $R\le Y$, let $0<\Delta\le 1$, and let $\theta>3$ be fixed. Set
    $\kappa:=0.001$. Then, for fixed $0<c<1$ and every $\eps>0$, one has
    \begin{multline*}
        J_{K,Y,R,\Delta} :=
        \#\Big\{ k_1,k_2\asymp K,\ n_1,n_2\asymp Y,\ r_1,r_2\asymp R,\ r_i<cn_i : \\
        \Big|
        k_1 \big(n_1^\theta\log n_1-(n_1-r_1)^\theta\log(n_1-r_1)\big)
        - k_2\big(n_2^\theta\log n_2-(n_2-r_2)^\theta\log(n_2-r_2)\big)
        \Big| \le \\
        \tilde c \Delta K Y^{\theta-1}R\log Y
        \Big\} \ll (KY)^{\eps}
        \Big( K^2Y^{2-\kappa} +
            \Delta K^2R^2Y^2 +
            KR^{\frac{1}{2}}Y^{\frac52-\kappa} +
            \Delta K R^{\frac{1}{2}}Y^{\frac92-\kappa} \Big)
    \end{multline*}
for any fixed $\tilde c>0$, where the implied constant may depend on $\theta,c,\tilde c$, and~$\eps$.
\end{proposition}

\begin{proof}
    We distinguish two cases according to the size of $\Delta$.

    \vspace{1ex}
    
    \textit{Case 1. $\Delta \ge Y^{-2-\kappa}$.} 
    Write
    \[
        f(n,r):=n^\theta \log n-(n-r)^\theta\log(n-r).
    \]
    Since $n\asymp Y$, $r\asymp R$, and $r<cn$, the mean-value theorem gives
    \[
        f(n,r)\asymp Y^{\theta-1}R\log Y.
    \]
    Therefore, if
    \[
        \big|k_1 f(n_1,r_1)-k_2 f(n_2,r_2)\big|
        \le \tilde c \Delta K Y^{\theta-1}R\log Y,
    \]
    then
    \[
        \bigg|
        \frac{k_1 f(n_1,r_1)}{k_2 f(n_2,r_2)}-1
        \bigg| \le C_1\Delta
    \]
    for some sufficiently large constant \(C_1=C_1(\theta)>0\). Hence, for
    \[
        \mathcal R:=\frac{k_1 f(n_1,r_1)}{k_2 f(n_2,r_2)},
    \] we have $|\log \mathcal R|\le C_2\Delta$
    with another constant \(C_2=C_2(\theta)>0\).

    Next, recall that the Fourier transform of $(\frac{\sin \pi x}{\pi x})^2$ is
    \(\max(0,1-|v|)\). By Fourier inversion,
    \begin{multline*}
        J_{K,Y,R,\Delta} \ll
        \sum_{k_1,k_2\asymp K}
        \sum_{n_1,n_2\asymp Y}
        \sum_{\substack{r_1,r_2\asymp R\\ r_i<c n_i}}
        \Big( \frac{\sin\big(\pi (\log \mathcal R)(2C_2\Delta)^{-1}\big)}
        {\pi (\log \mathcal R)(2C_2\Delta)^{-1}}
        \Big)^2 \\
        \ll \int_{-1}^1 \max(0,1-|v|)
        \bigg| \sum_{k\asymp K} \sum_{n\asymp Y}
        \sum_{\substack{r\asymp R\\ r<cn}}
        e\Big( v\frac{\log k+\log f(n,r)}{C_2\Delta}
        \Big) \bigg|^2 dv.
    \end{multline*}
    Using positivity of the weight and symmetry, we get
    \[
        J_{K,Y,R,\Delta} \ll
        \int_0^1 \bigg|
        \sum_{k\asymp K} e\Big(v\frac{\log k}{C_2\Delta}\Big) \bigg|^2 \bigg|
        \sum_{n\asymp Y}\sum_{\substack{r\asymp R\\ r<cn}} 
        e\Big(v\frac{\log f(n,r)}{C_2\Delta}\Big)
        \bigg|^2 dv.
    \]
    After the substitution \(u=v/(C_2\Delta)\), this becomes
    \[
        J_{K,Y,R,\Delta} \ll I_{K,Y,R,\Delta},
    \]
    where
    \[
        I_{K,Y,R,\Delta} := \Delta
        \int_0^{\Delta^{-1}}
        \bigg| \sum_{k\asymp K} e(u\log k)
        \bigg|^2 \bigg|
        \sum_{n\asymp Y}\sum_{\substack{r\asymp R\\ r<cn}} e(u\log f(n,r)) \bigg|^2 du.
    \]

    We split this integral dyadically as follows:
    \[
        I_{K,Y,R,\Delta}^{(0)} :=
        \Delta\int_0^1 |\ldots|^2 |\ldots|^2\,du,
        \qquad I_{K,Y,R,\Delta}^{(T)} :=
        \Delta\int_T^{2T} |\ldots|^2 |\ldots|^2\,du,
    \]
    where \(T=1,2,4,\ldots\), and \(T\le \Delta^{-1}\le Y^{2+\kappa}\) (the last integral may in general run from $T$ to $(1+c')T$ for some $c'>0$).

    The contribution of \(I_{K,Y,R,\Delta}^{(0)}\) is bounded trivially:
    \begin{equation} \label{bound_1_trivial_large}
        I_{K,Y,R,\Delta}^{(0)} \ll \Delta K^2R^2Y^2.
    \end{equation}

    Next, assume that \(1\le T\le c_0 Y\) with sufficiently small $c_0 > 0$. Then Lemma~\ref{Kusmin_Landau} applies to the \(n\)-sum, uniformly in \(r\), since
    \[
        \frac{d}{dn}\log f(n,r)\asymp \frac{1}{Y},
    \] and this derivative is monotone. Indeed, for the latter condition, it is enough to show that the function $f(x,r)$ is log-concave for fixed $r$. It is therefore sufficient to show that 
    \[
        g(x) := \frac{d}{dx} x^{\theta} \log x \qquad \text{is log-concave.}
    \] This follows from
    \[
        \frac{d^2}{dx^2} \log g(x) = \frac{g(x) g''(x) - g'(x)^2}{g(x)^2} < 0 \qquad \text{for all sufficiently large $x$},
    \] which can be verified directly.
    Thus
    \[
        \sum_{r\asymp R}\sum_{\substack{n\asymp Y\\ r<cn}} e(u\log f(n,r))
        \ll R\cdot \frac{Y}{T},
    \]
    and therefore
    \begin{equation} \label{bound_2_KL_large}
        \Delta
        \sum_{\substack{T\text{ dyadic}\\ 1\le T\le c_0 Y}} \int_T^{2T}
        K^2 \cdot R^2 \Big(\frac{Y}{T}\Big)^2\,du
        \ll \Delta K^2R^2Y^2.
    \end{equation}

    Similarly, if \(1\le T\le c_0 K\), then Lemma~\ref{Kusmin_Landau} applies to the \(k\)-sum, and we obtain
    \begin{equation} \label{bound_2_KL_dop_large}
        \Delta \sum_{\substack{T\text{ dyadic}\\ 1\le T\le c_0 K}}
        \int_T^{2T}
        \Big(\frac{K}{T}\Big)^2 R^2Y^2\,du
        \ll \Delta K^2R^2Y^2.
    \end{equation}

    It remains to consider the range in which
    \[
        T>c_0 \max\big(K,\,Y\big).
    \]
    In this range we apply Hölder's inequality:
    \begin{multline*}
        I_{K,Y,R,\Delta}^{(T)}
        \le \Delta \bigg( \int_T^{2T}
        \Big| \sum_{k\asymp K} e(u\log k) \Big|^8 du \bigg)^{1/4} \cdot
        \\
        \bigg( \int_T^{2T} \Big|
        \sum_{n\asymp Y}\sum_{\substack{r\asymp R\\ r<cn}} e(u\log f(n,r))
        \Big|^{8/3} du \bigg)^{3/4}
        =:\Delta\, S_1^{\frac14} S_2^{\frac34},
    \end{multline*} where $S_1$ and $S_2$ denote the corresponding integrals.

    By Lemma~\ref{MVT_zeta},
    \begin{equation} \label{S_1_bound_large}
        S_1 \ll K^{4+\eps}T^{\frac32+\eps}.
    \end{equation}

    Next, write
    \begin{equation} \label{S_2_decomposition_large}
        S_2 \le \bigg(
        \max_{T\le u\le 2T} \Big|
        \sum_{r\asymp R}\sum_{\substack{n\asymp Y\\ r<cn}} e(u\log f(n,r))
        \Big|^{2/3} \bigg) \int_T^{2T} \Big|
        \sum_{r\asymp R}\sum_{\substack{n\asymp Y\\ r<cn}} e(u\log f(n,r)) \Big|^2 du
        =: S_3 S_4.
    \end{equation}

    To estimate \(S_3\), we apply Lemma~\ref{exp_pairs} with the second bound, and then estimate the sum over $r$ trivially. This gives
    \begin{equation} \label{S_3_bound_large}
        S_3 \ll \big(R\cdot Y^{\frac12+\eps}T^{\frac{13}{84}+\eps}\big)^{\frac{2}{3}}
        \ll R^{\frac23} Y^{\frac13} T^{\frac{13}{126}+\eps}
        \ll R^{\frac23} Y^{\frac{34}{63}+\frac{13\kappa}{126}+\eps},
    \end{equation} since \(T\le \Delta^{-1}\le Y^{2+\kappa}\).

    We now estimate \(S_4\). Expanding the square, we get
    \[
        S_4 = \int_T^{2T}
        \sum_{\substack{n_1,n_2\asymp Y\\ r_1,r_2\asymp R\\ r_i<n_i}}
        e\Big(u\log \frac{f(n_1,r_1)}{f(n_2,r_2)}\Big)\,du.
    \]
    As at the beginning of the proof, the Fejér kernel argument gives
    \[
        S_4 \ll T\, \#\Big\{
        n_1,n_2\asymp Y,\ r_1,r_2\asymp R,\ r_i<cn_i : \Big| \log \frac{f(n_1,r_1)}{f(n_2,r_2)} \Big| \le \frac{1}{T} \Big\}.
    \]
    Since \(f(n,r)\asymp Y^{\theta-1}R\log Y\), this implies
    \[
        S_4 \ll T\, \#\Big\{
        n_1,n_2\asymp Y,\ r_1,r_2\asymp R,\ r_i<cn_i : |f(n_1,r_1)-f(n_2,r_2)| \ll
        \frac{1}{T}Y^{\theta-1}R\log Y \Big\}.
    \]
    Enlarging the set symmetrically, we bound this by 
    \[
        T\, \#\Big\{
        n_1,n_2,n_3,n_4\asymp Y :
        \big| n_1^\theta\log n_1-n_2^\theta\log n_2 - n_3^\theta\log n_3 + n_4^\theta\log n_4
        \big| \ll
        \frac{1}{T}Y^{\theta-1}R\log Y \Big\}.
    \] This is valid because of the initial restriction $r_i < cn_i$ (and this is a crucial difference with the off-diagonal counting problem). By~\cite[Lemma~1]{Robert-Sargos}, this is bounded by
    \[
        \frac{R}{Y}\int_0^{bTY/R}
        \Big|
        \sum_{n\asymp Y} e\Big(v\frac{n^\theta\log n}{Y^\theta\log Y}\Big)
        \Big|^4 dv
    \]
    for some absolute constant \(b>0\). Then, by Corollary~\ref{energy_bound}, we have
    \begin{equation} \label{S_4_bound_large}
        S_4 \ll RY^3 + TY^{2.45} \ll Y^{4.45+\kappa},
    \end{equation}
    since \(R\le Y\) and \(T\le Y^{2+\kappa}\).

    Combining~\eqref{S_1_bound_large}, \eqref{S_2_decomposition_large}, \eqref{S_3_bound_large}, and~\eqref{S_4_bound_large}, we obtain
    \begin{multline} \label{last_ineq}
        I_{K,Y,R,\Delta}^{(T)}
        \ll \Delta K^{1+\eps}
        T^{\frac38+\eps} \cdot
        R^{\frac12} Y^{\frac{17}{42}+\frac{13\kappa}{168}+\eps} \cdot
        Y^{\frac{267}{80}+\frac{3\kappa}{4}} \\
        \ll \Delta (KY)^{\eps} K R^{\frac12}
        Y^{\frac{7547}{1680}+\frac{101\kappa}{84}}
        \ll \Delta (KY)^{\eps}
        K R^{\frac12} Y^{\frac92-\kappa},
    \end{multline}
    by our choice of \(\kappa\).

    Summing over dyadic \(T\) and combining the result with
    \eqref{bound_1_trivial_large}, \eqref{bound_2_KL_large}, \eqref{bound_2_KL_dop_large}, and~\eqref{last_ineq}, we conclude that
    \begin{equation} \label{first_case_final_large}
        J_{K,Y,R,\Delta} \ll I_{K,Y,R,\Delta}
        \ll \Delta K^2R^2Y^2 +
        \Delta (KY)^{\eps} K R^{\frac12} Y^{\frac92-\kappa}.
    \end{equation}

    \vspace{2ex}

    \textit{Case 2. $\Delta < Y^{-2-\kappa}$.}
    Let $\tilde \Delta:=Y^{-2-\kappa}$. By monotonicity,
    \[
        J_{K,Y,R,\Delta}\le J_{K,Y,R,\tilde \Delta}.
    \]
    Applying~\eqref{first_case_final_large} with \(\Delta=\tilde \Delta\), we get
    \begin{equation} \label{second_case_final_large}
        J_{K,Y,R,\Delta} \ll
        (KY)^{\eps} \Big(
            \tilde \Delta K^2 R^2 Y^2+
            \tilde \Delta K R^{\frac12} Y^{\frac92-\kappa} \Big) \ll (KY)^{\eps}
        \Big( K^2 Y^{2-\kappa} +
            K R^{\frac12}Y^{\frac52-\kappa}
        \Big),
    \end{equation}
    since \(R\le Y\). Combining \eqref{first_case_final_large} and \eqref{second_case_final_large} completes the proof.
\end{proof}

\vspace{2ex}

\begin{proposition} \label{count_prop_large_theta_nondiag}
    Let $K,Y_1,Y_2\ge 1$ be real numbers such that $Y_2 \ge Y_1$, let $0<\Delta\le 1$, and let $3 < \theta<7$ be fixed. Set $\kappa:=0.001$. Then, for fixed $0<c<1$ and every $\eps>0$, one has
    \begin{multline*}
        J_{K,Y_1,Y_2,\Delta} :=
        \#\Big\{
        k_1,k_2\asymp K,\ n_1,n_3\asymp Y_1,\ n_2,n_4\asymp Y_2,\ n_1 < cn_2, n_3 < cn_4: \\
        \Big|
        k_1\Big(n_2^\theta\log n_2-n_1^\theta\log n_1\Big)
        - k_2\Big(n_4^\theta\log n_4-n_3^\theta\log n_3\Big)
        \Big| \le
        \tilde c \Delta K Y_2^{\theta} \log Y_2
        \Big\} \ll \\ (KY_2)^{\eps}
        \Big( K^2 Y_2^{2-\kappa} + \Delta K^2 Y_2^4 + KY_2^{3-\kappa} + \Delta K Y_2^{5-\kappa} \Big)
    \end{multline*} for any fixed $\tilde c>0$, where the implied constant may depend on $\theta,c,\tilde c$, and $\eps$.
\end{proposition}

\begin{proof}
    The proof is almost identical to the previous one. We focus on the main differences. \\

    \textit{Case 1. $\Delta \ge Y_2^{-2-\kappa}$.} Similarly, for $m\asymp Y_1$, $n\asymp Y_2$, and $m<cn$, we introduce the function
    \[
        f(m,n) := n^{\theta} \log n - m^{\theta} \log m, \qquad f(m, n) \asymp Y_2^{\theta} \log Y_2. 
    \] Repeating the same steps, we reduce the problem to estimating the integral
    \[
        I_{K, Y_1, Y_2, \Delta} := \Delta \int_0^{\Delta^{-1}} \bigg| \sum_{k \asymp K} e\big( u \log k \big) \bigg|^2 \bigg| \sum_{m \asymp Y_1} \sum_{n \asymp Y_2} e\big( u \log f(m,n) \big) \bigg|^2 du,
    \] and similarly split it into the integrals $I_{K,Y_1, Y_2, \Delta}^{(0)}$ and $I_{K,Y_1, Y_2, \Delta}^{(T)}$ for $T = 1,2,4,\ldots$, $T \le \Delta^{-1} \le Y_2^{2+\kappa}$, defined similarly to the previous case. 

    The first integral $I_{K,Y_1, Y_2, \Delta}^{(0)}$ is evaluated trivially:
    \begin{equation} \label{the_first_trivial}
        I_{K,Y_1, Y_2, \Delta}^{(0)} \ll \Delta K^2 Y_1^2 Y_2^2.
    \end{equation} When $1 \le T \le c_0 \max(Y_2,K)$, with sufficiently small $c_0>0$, we apply Lemma~\ref{Kusmin_Landau} to the $n$- or $k$-sum, depending on which one is longer. This gives the contribution
    \begin{equation} \label{the_second_Kusmin_Landau}
        \Delta \sum_{\substack{T \text{ dyadic} \\ 1 \le T \le c_0 \max(K, Y_2)}} \int_T^{2T} \min(K^2, Y_2^2) Y_1^2 \Big( \frac{\max(K, Y_2)}{T} \Big)^2 du \ll \Delta K^2 Y_1^2 Y_2^2.
    \end{equation}

    In the remaining range, we again apply H\"older's inequality:
    \[
        I_{K, Y_1, Y_2, \Delta}^{(T)} \le \Delta S_1^{\frac14} S_2^{\frac34},
    \] where $S_1$ and $S_2$ have the same meaning as in the previous proposition. The sum $S_1$ is the same as in the previous proposition and is bounded by
    \[
        S_1 \ll K^{4+\eps} T^{\frac32+\eps}.
    \] Next,
    \[
        S_2 \le \bigg( \max_{T \le u \le 2T} \Big| \sum_{m \asymp Y_1} \sum_{n \asymp Y_2} e\big( u \log f(m,n) \big) \Big|^{2/3} \bigg) \int_T^{2T} \Big| \sum_{m \asymp Y_1} \sum_{n \asymp Y_2} e\big( u \log f(m,n) \big) \Big|^2 du =: S_3 S_4.
    \] Similarly, by Lemma~\ref{exp_pairs},
    \[
        S_3 \ll Y_1^{\frac23} Y_2^{\frac{34}{63} + \frac{13\kappa}{126} + \eps},
    \] and for $S_4$ we obtain
    \[
        S_4 \ll T \# \Big\{ m_1, m_2 \asymp Y_1, \ n_1, n_2 \asymp Y_2: | f(m_1, n_1) - f(m_2, n_2) | \ll \frac{1}{T} Y_2^{\theta} \log Y_2 \Big\}.
    \] Here we do not enlarge the counting quantity in order to reduce it to a fourth-moment problem; rather, we keep both pairs of variables localized in their respective ranges. By a slight modification of~\cite[Lemma~1]{Robert-Sargos}, we then get
    \begin{multline*}
        S_4 \ll \int_0^T \bigg| \sum_{\substack{m_1, m_2 \asymp Y_1 \\ n_1, n_2 \asymp Y_2}} e \Big( u \frac{n_1^{\theta} \log n_1 - m_1^{\theta} \log m_1 + n_2^{\theta} \log n_2 - m_2^{\theta} \log m_2}{Y_2^{\theta} \log Y_2} \Big) \bigg| du \ll \\
        \int_0^T \bigg| \sum_{m \asymp Y_1} e\Big( u \frac{m^{\theta} \log m}{Y_2^{\theta} \log Y_2} \Big) \bigg|^2 \bigg| \sum_{n \asymp Y_2} e\Big( u \frac{n^{\theta} \log n}{Y_2^{\theta} \log Y_2} \Big) \bigg|^2 du.
    \end{multline*}

    Finally we consider two more subcases depending on the mutual sizes of $Y_1$ and $Y_2$. \\

    \textit{Case 1.1. $Y_1 \le Y_2^{67/84-3\kappa}$.} 
    Here, since the sum over $m \asymp Y_1$ is relatively short, we can use the first bound in Lemma~\ref{unbalanced_energy_bound}. It gives
    \[
        S_4 \ll Y_2^{\eps} \Big( Y_1^2 Y_2^2 + T Y_1^2 Y_2 \Big). 
    \]

    Combining it with the bounds for $S_1$ and $S_3$ from above, we find
    \begin{multline*}
        I_{K,Y_1,Y_2,\Delta}^{(T)} \ll \Delta K^{1+\frac{\eps}{4}} T^{\frac{3}{8}+\frac{\eps}{4}} \cdot Y_1^{\frac{1}{2}} Y_2^{\frac{17}{42} + \frac{13\kappa}{168} + \frac{3\eps}{4}} \cdot Y_2^{\frac{3\eps}{4}} \Big( Y_1^2 Y_2^2 + T Y_1^2 Y_2 \Big)^{\frac{3}{4}} \ll \\
        \Delta K^{1+\frac{\eps}{4}} Y_1^2 Y_2^{\frac{143}{42} + \frac{101\kappa}{84}} \ll \Delta K^{1+\frac{\eps}{4}} Y_2^{5-2\kappa},
    \end{multline*} where we used $T \ll Y_2^{2+\kappa}$ and $Y_1 \ll Y_2^{67/84 - 3\kappa}$. Summation over dyadic $T$ contributes at most a factor of $\log Y_2$. Therefore, using~\eqref{the_first_trivial} and~\eqref{the_second_Kusmin_Landau}, we find
    \[
        J_{K,Y_1,Y_2,\Delta} \ll \Delta K^2 Y_1^2 Y_2^2 + \Delta K^{1+\eps} Y_2^{5-\kappa}.
    \]

    \textit{Case 1.2. $Y_1 > Y_2^{67/84-3\kappa}$.} Here we use the second bound of Lemma~\ref{unbalanced_energy_bound}. Assume first that $T \le c_0 Y_2^{\theta} / Y_1^{\theta-1}$. Then
    \[
        S_4 \ll Y_2^{\eps} \Big( Y_1^2 Y_2^2 + \frac{Y_2^{2\theta}}{Y_1^{2\theta-2}} \Big).
    \] In this case,
    \begin{multline*}
        I_{K, Y_1, Y_2, \Delta}^{(T)} \ll \Delta K^{1+\frac{\eps}{4}} T^{\frac{3}{8}+\frac{\eps}{4}} \cdot Y_1^{\frac{1}{2}} Y_2^{\frac{17}{42} + \frac{13\kappa}{168} + \frac{3\eps}{4}} \cdot Y_2^{\frac{3\eps}{4}} \bigg( (Y_1 Y_2)^{\frac{3}{2}} + \frac{Y_2^{3\theta/2}}{Y_1^{3\theta/2 - 3/2}} \bigg) \ll \\
        (KY_2)^{\eps} \cdot \Delta K \Big( \frac{Y_2^{\theta}}{Y_1^{\theta-1}} \Big)^{3/8} \Big(  Y_1^2 Y_2^{\frac{40}{21} + \frac{13\kappa}{168}} + Y_1^{2 - \frac{3\theta}{2}} Y_2^{\frac{17}{42} + \frac{3\theta}{2} + \frac{13\kappa}{168}} \Big).
    \end{multline*} The maximum of the last expression is achieved at $\theta = 7$ and $Y_1 = Y_2^{67/84 - 3\kappa}$. Thus,
    \[
        I_{K, Y_1, Y_2, \Delta}^{(T)} \ll (KY_2)^{\eps} \cdot \Delta K \big( Y_2^{\frac{485}{112} + \kappa} + Y_2^{\frac{39}{8} + 7\kappa} \big) \ll \Delta K^{1+\eps} Y_2^{5-\kappa},
    \] which leads to the same bound as in Case~1.1.

    Finally, suppose that $T > c_0 Y_2^{\theta} / Y_1^{\theta-1}$. Then
    \[
        S_4 \ll Y_2^{\eps} \Big( 1 + \frac{T}{Y_2^{1.55}} \Big) \Big( Y_1^2 Y_2^2 + Y_1^{\frac{374}{641} + \frac{356\theta}{641}} Y_2^{3.411 - \frac{356\theta}{641}} \Big),
    \] and, consequently,
    \[
        I_{K, Y_1, Y_2, \Delta}^{(T)} \ll \Delta K^{1+\frac{\eps}{4}} T^{\frac{3}{8}+\frac{\eps}{4}} \cdot Y_1^{\frac{1}{2}} Y_2^{\frac{17}{42} + \frac{13\kappa}{168} + \frac{3\eps}{2}} \Big( 1 + \frac{T}{Y_2^{1.55}} \Big)^{3/4} \Big( Y_1^2 Y_2^2 + Y_1^{\frac{374}{641} + \frac{356\theta}{641}} Y_2^{3.411 - \frac{356\theta}{641}} \Big)^{3/4}.
    \] The maximum of the last expression is achieved when $T = Y_2^{2+\kappa}$ and $Y_1 = Y_2$. This gives
    \[
        I_{K, Y_1, Y_2, \Delta}^{(T)} \ll \Delta K^{1+\eps} Y_2^{4.993 + 2\kappa} \ll \Delta K^{1+\eps} Y_2^{5-\kappa},
    \] which is the same as in the previous cases. It remains only to sum over the dyadic values of $T$ in the corresponding range. We conclude
    \[
        J_{K,Y_1,Y_2, \Delta} \ll \Delta K^2 Y_1^2 Y_2^2 + \Delta K^{1+\eps} Y_2^{5-\kappa} \ll \Delta K^2 Y_2^4 + \Delta K^{1+\eps} Y_2^{5-\kappa}.
    \]

    \vspace{2ex}

    \textit{Case 2. $\Delta < Y_2^{-2-\kappa}$.} Here we again set $\tilde \Delta := Y_2^{-2-\kappa}$ and, by monotonicity, obtain
    \[
        J_{K,Y_1,Y_2,\Delta} \le J_{K,Y_1,Y_2,\tilde \Delta} \ll K^2 Y_2^{2-\kappa} + K^{1+\eps} Y_2^{3-2\kappa}.
    \] This completes the proof.

\end{proof}

\section{Counting for small $\theta$}

In this section, we prove the counting estimates used in the proof when \(0<\theta<\frac35\). We start with the near-diagonal case.

\begin{proposition} \label{count_small_theta_diag}
    Let \(K,Y,R\ge 1\) be real numbers, let $0 < \Delta \le 1$, \(0 < \theta < \frac{3}{5}\), and let \(0<c_1,c_2,c_3<1\) be fixed. Furthermore, assume that there is an absolute constant \(c_4>0\) such that $R\le c_4 K Y^{\theta-1}$.
    Then, for every $\eps>0$, we have
    \begin{multline} \label{to_prove_small_theta_new}
        J_{K,Y,R,\Delta} :=
        \# \Big\{
        c_1 RY K^{-1}<\ell_1,\ell_2\le RY K^{-1},\
        c_2 K Y^{\theta-1}<m_1,m_2\le K Y^{\theta-1}, \\
        c_3 R<r_1,r_2\le R,\ r_i\le 100m_i:
        \Big| F(\ell_1,\eta_1,\theta)\ell_1^{\frac1\theta}\eta_1^{1-\frac1\theta} -F(\ell_2,\eta_2,\theta)\ell_2^{\frac1\theta}\eta_2^{1-\frac1\theta}
        \Big| \le c_5 \Delta RY \log(RY K^{-1})
        \Big\} \ll \\
        (K Y R)^{\eps}
        \Big(
            K^2 Y^{2\theta-2}R^2+
            K^{\frac43}Y^{\frac{8\theta}{3}-\frac43}R^2+
            \Delta K^2 Y^{4\theta-2}R^2
        \Big)
    \end{multline}
for any fixed $c_5>0$, where the implied constant may depend on $\theta,c_1,c_2,c_3,c_4,c_5$, and $\eps$, where
    \begin{gather*}
        \eta_i:=\eta_i(m_i,r_i)=m_i^{-\frac{\theta}{1-\theta}}-(m_i+r_i)^{-\frac{\theta}{1-\theta}}, \\
        F(\ell,\eta,\theta) :=
        D(\theta)+\frac{1}{\theta^2}\log \ell-\frac{1}{\theta^2}\log \eta
        -\Big(1-\frac{1}{\theta}\Big)\frac{\eta_{\theta}}{\eta},
    \end{gather*}
    and \(D=D(\theta)\) is a constant.
\end{proposition}

\begin{proof}
As in the previous section, we consider two cases depending on the size of~$\Delta$.

\vspace{2ex}

\textit{Case 1. \(\Delta \ge K^{-2/3}Y^{-4\theta/3+2/3}\).}

We may assume that \(K\le (YR)^{1-\eps}\), since otherwise the number of choices for \(\ell_1,\ell_2\) is \(\ll (RY)^{\eps}\), and the result follows by the trivial bound. We may also assume that \(K Y^{\theta-1}\gg 1\), since otherwise \(J_{K,Y,R,\Delta}=0\). Finally, recall that
\begin{equation} \label{first_recall_new}
    F(\ell,\eta,\theta)\asymp \log \ell
\end{equation}
in the present ranges.

The next steps follow the same pattern as in the large-\(\theta\) case. We divide the inequality inside the definition of \(J_{K,Y,R,\Delta}\) in~\eqref{to_prove_small_theta_new} by the second term. This gives
\[
    \bigg| \frac{F(\ell_1,\eta_1,\theta)\ell_1^{\frac1\theta}\eta_1^{1-\frac1\theta}} {F(\ell_2,\eta_2,\theta)\ell_2^{\frac1\theta}\eta_2^{1-\frac1\theta}}-1
    \bigg| \le c_6\Delta
\] 
for some absolute \(c_6>0\), which implies
\[
    \Big| \log \frac{F(\ell_1,\eta_1,\theta)\ell_1^{\frac1\theta}\eta_1^{1-\frac1\theta}} {F(\ell_2,\eta_2,\theta)\ell_2^{\frac1\theta}\eta_2^{1-\frac1\theta}} \Big| \le c_7\Delta
\]
for some \(c_7>0\). Thus
\[
    J_{K,Y,R,\Delta} \ll
    \Delta \int_0^{\Delta^{-1}}
    \Big|
        \sum_{\ell,m,r} e\big(\Phi(\ell;v,\eta,\theta)\big)
    \Big|^2 dv,
\]
where
\begin{equation} \label{def_Phi_new}
    \Phi(\ell;v,\eta,\theta) :=
    -v\Big(
        \log F(\ell,\eta,\theta)
        +\frac{1}{\theta}\log \ell
        +\Big(1-\frac{1}{\theta}\Big)\log \eta
    \Big).
\end{equation}
We split the integral into
\[
    I_{K,Y,R,\Delta}^{(0)} :=
    \Delta \int_0^1 |\ldots|^2\,dv, \qquad
    I_{K,Y,R,\Delta}^{(T)} :=
    \Delta \int_T^{2T} |\ldots|^2\,dv,
    \quad T=1,2,4,\ldots.
\]

The integral over \([0,1]\) is bounded trivially:
\begin{equation} \label{trivial_integral_new}
    I_{K,Y,R,\Delta}^{(0)} \ll
    \Delta (RYK^{-1})^2 (K Y^{\theta-1})^2 R^2 \ll
    \Delta Y^{2\theta}R^4 \ll
    \Delta K^2 Y^{4\theta-2}R^2,
\end{equation}
since \(R^2\le K^2 Y^{2\theta-2}\).

Set
\[
    L:=RYK^{-1}, \qquad M:=K Y^{\theta-1}.
\]

Next, to estimate the integrals $I_{K,Y,R,\Delta}^{(T)}$, we need several derivatives of \(\eta\), \(F\), and the phase \(\Phi\). Recall that
\[
    \eta(m,r) =
    m^{-\frac{\theta}{1-\theta}}-(m+r)^{-\frac{\theta}{1-\theta}} =
    m^{-\frac{\theta}{1-\theta}}
    \Big(
        1-\Big(\frac{m}{m+r}\Big)^{\frac{\theta}{1-\theta}}
    \Big).
\]
Differentiating once and twice with respect to \(m\), we obtain
\[
    \eta_m =
    -\frac{\theta}{1-\theta}m^{-\frac{1}{1-\theta}}
    +\frac{\theta}{1-\theta}(m+r)^{-\frac{1}{1-\theta}} =
    -\frac{\theta}{1-\theta}m^{-\frac{1}{1-\theta}}
    \Big(
        1-\Big(\frac{m}{m+r}\Big)^{\frac{1}{1-\theta}}
    \Big),
\] and
\[
    \eta_{mm} =
    \frac{\theta}{(1-\theta)^2}m^{-\frac{2-\theta}{1-\theta}}
    -\frac{\theta}{(1-\theta)^2}(m+r)^{-\frac{2-\theta}{1-\theta}} =
    \frac{\theta}{(1-\theta)^2}m^{-\frac{2-\theta}{1-\theta}}
    \Big(
        1-\Big(\frac{m}{m+r}\Big)^{\frac{2-\theta}{1-\theta}}
    \Big).
\]
Therefore
\begin{equation} \label{eta_m_over_eta_new}
    \frac{\eta_m}{\eta} =
    -\frac{\theta}{1-\theta}\frac{1}{m}\,
    \frac{1-\big(\frac{m}{m+r}\big)^{1/(1-\theta)}
    }{
        1-\big(\frac{m}{m+r}\big)^{\theta/(1-\theta)}
    }.
\end{equation}
Since \(r\le 100m\), we have
\[
    \frac{m}{m+r}\in \Big[\frac{1}{101},1\Big),
\]
and hence the large fraction in~\eqref{eta_m_over_eta_new} is a continuous positive function of $\frac{m}{m+r}$ on \([\frac{1}{101},1)\). Therefore
\begin{equation} \label{eta_m_over_eta_estimate_new}
    \frac{\eta_m}{\eta}\asymp \frac{1}{m}.
\end{equation}
Also,
\[
    \frac{\eta_{mm}}{\eta} =
    \frac{\theta}{(1-\theta)^2}\frac{1}{m^2}\,
    \frac{
        1-\big(\frac{m}{m+r}\big)^{(2-\theta)/(1-\theta)}
    }{
        1-\big(\frac{m}{m+r}\big)^{\theta/(1-\theta)}
    }.
\]
Thus
\begin{multline*}
    (\log \eta)_{mm} =
    \frac{\eta_{mm}}{\eta} -
    \Big(\frac{\eta_m}{\eta}\Big)^2 = \\
    \frac{\theta}{(1-\theta)^2}\frac{1}{m^2}\,
    \frac{
        \Big(1-\big(\frac{m}{m+r}\big)^{(2-\theta)/(1-\theta)}\Big)
        \Big(1-\big(\frac{m}{m+r}\big)^{\theta/(1-\theta)}\Big) -
        \theta\Big(1-\big(\frac{m}{m+r}\big)^{1/(1-\theta)}\Big)^2
    }{
        \Big(1-\big(\frac{m}{m+r}\big)^{\theta/(1-\theta)}\Big)^2
    }.
\end{multline*}
Again, the large fraction above is a positive continuous function of \(\frac{m}{m+r}\in [\frac{1}{101},1)\). Therefore
\begin{equation} \label{estimate_for_log_eta_new}
    (\log \eta)_{mm}\asymp_{\theta} \frac{1}{m^2}.
\end{equation}

Next, recall that
\[
    F(\ell,\eta,\theta) =
    D(\theta)+\frac{1}{\theta^2}\log \ell-\frac{1}{\theta^2}\log \eta
    -\Big(1-\frac{1}{\theta}\Big)\frac{\eta_{\theta}}{\eta}.
\]
We have
\[
    \eta_{\theta} =
    -\frac{\log m}{(1-\theta)^2}m^{-\frac{\theta}{1-\theta}} +
    \frac{\log(m+r)}{(1-\theta)^2}(m+r)^{-\frac{\theta}{1-\theta}},
\]
and therefore
\[
    \frac{\eta_{\theta}}{\eta} =
    -\frac{1}{(1-\theta)^2}
    \bigg(
        \log m + \frac{
            \big(\frac{m}{m+r}\big)^{\theta/(1-\theta)} \log\! \big(\frac{m}{m+r} \big)
        }{
            1-\big( \frac{m}{m+r} \big)^{\theta/(1-\theta)}
        }
    \bigg).
\]

Next, we differentiate with respect to \(m\). We have
\[
    \frac{d}{dm}\Big(\frac{m}{m+r}\Big) \ll \frac{1}{m}, \qquad
    \frac{d^2}{dm^2}\Big(\frac{m}{m+r}\Big) \ll \frac{1}{m^2}.
\]
One then directly verifies
\[
    \Big(\frac{\eta_{\theta}}{\eta}\Big)_m = O\Big(\frac{1}{m}\Big),
    \qquad
    \Big(\frac{\eta_{\theta}}{\eta}\Big)_{mm} = O\Big(\frac{1}{m^2}\Big).
\]
Combining this with~\eqref{eta_m_over_eta_estimate_new} and~\eqref{estimate_for_log_eta_new}, we get
\begin{equation} \label{F_m_estimate_new}
    F_m = -\frac{1}{\theta^2}\frac{\eta_m}{\eta}
    -\Big(1-\frac{1}{\theta}\Big)\Big(\frac{\eta_{\theta}}{\eta}\Big)_m = O\Big(\frac{1}{m}\Big),
\end{equation}
and similarly
\begin{equation} \label{F_mm_estimate_new}
    F_{mm}=O\Big(\frac{1}{m^2}\Big).
\end{equation}

\vspace{2ex}

\textit{Range 1: \(1\le T\le c_0 M\).}

Here we assume that $c_0>0$ is sufficiently small, so that Lemma~\ref{Kusmin_Landau} is applicable. By~\eqref{first_recall_new}, \eqref{eta_m_over_eta_estimate_new}, and \eqref{F_m_estimate_new}, we obtain
\[
    \frac{d}{dm}\Phi(\ell;v,\eta,\theta)
    = -v\Big( \frac{F_m}{F} +
        \Big(1-\frac{1}{\theta}\Big)\frac{\eta_m}{\eta} \Big)
    \asymp \frac{T}{m}
    \asymp \frac{T}{M}.
\]

Since \(1\le T\le c_0 M\), Lemma~\ref{Kusmin_Landau} applies to the sum over \(m\). We also note that $\Phi_m$ is monotone in $m$. Indeed, writing
$a=\theta/(1-\theta)$ and
\[
    \eta(m,r)=m^{-a}-(m+r)^{-a},
\]
one has
\[
    \frac{\eta_m}{\eta}\asymp -\frac1m, \qquad
    \Big(\frac{\eta_m}{\eta}\Big)_m
    =(\log\eta)_{mm}\asymp \frac1{m^2}.
\]
Since $1-\frac{1}{\theta}<0$, it follows that
\[
    \frac{d}{dm}\Big(\Big(1-\frac1\theta\Big)\frac{\eta_m}{\eta}\Big)
    \asymp -\frac1{m^2}.
\]
On the other hand, $F_m=O(\frac{1}{m})$, $F_{mm}=O(\frac{1}{m^2})$, and
$F\asymp \log Y$, whence
\[
    \frac{d}{dm}\Big(\frac{F_m}{F}\Big)
    =O\Big(\frac1{m^2\log Y}\Big).
\]
Thus $\Phi_{mm}$ has a fixed sign, and hence $\Phi_m$ is monotone. Then Lemma~\ref{Kusmin_Landau} gives
\[
    I_{K,Y,R,\Delta}^{(T)} \ll \Delta T
    \Big( \sum_{\ell,r} \frac{M}{T} \Big)^2
    \ll \Delta T
    \Big( \frac{RY}{K}\cdot R\cdot \frac{M}{T}
    \Big)^2.
\]
Summation over dyadic \(T\) in this range gives
\begin{equation} \label{Kusmin_Landau_integral_new}
    \sum_{\substack{T \text{ dyadic} \\ 1\le T\le c_0 M}} I_{K,Y,R,\Delta}^{(T)} \ll
    \Delta L^2R^2M^2 \ll
    \Delta Y^{2\theta}R^4 \ll
    \Delta K^2Y^{4\theta-2}R^2.
\end{equation}

\vspace{2ex}

\textit{Range 2: \(c_0 M \le T\ll \Delta^{-1}\).}

Here we apply Poisson summation to the sum over \(\ell\), and then apply Lemma~\ref{Corput} to the sum over \(m\). Since $K \ll (YR)^{1-\eps}$, we have
\[
    F(\ell,\eta,\theta)\gg \log \ell \gg_{\eps}\log Y,
\]
and one easily checks that
\[
    F_{\ell}\asymp \frac{1}{\ell}\asymp \frac{1}{L}, \qquad
    F_{\ell\ell}\asymp -\frac{1}{\ell^2}\asymp -\frac{1}{L^2}.
\]
Hence
\[
    \Phi_{\ell\ell}(\ell;v,\eta,\theta) =
    -v\Big( \frac{F_{\ell\ell}}{F}-\frac{F_{\ell}^2}{F^2}\Big)+\frac{v}{\theta \ell^2} \asymp \frac{T}{L^2}.
\]
Similarly,
\[
    \Phi_{\ell\ell\ell}\asymp \frac{T}{L^3}, \qquad
    \Phi_{\ell\ell\ell\ell}\asymp \frac{T}{L^4}.
\]

Applying~\cite[Theorem~8.16]{IK}, we get
\begin{equation} \label{IK_816_new}
    \sum_{\ell\asymp L} e\big(\Phi(\ell;v,\eta,\theta)\big) = \sum_{a<n<b}
    \frac{1}{(\Phi_{\ell\ell}(u_0))^{1/2}}
    e\Big(\Phi(u_0)-nu_0+\frac{1}{8}\Big) +
    R_{\Phi}(K,Y,R),
\end{equation}
where \(u_0=u_0(n,m,r,v)\),
\begin{gather}
    \Phi_{\ell}(u_0,m)=n, \label{def_n_new} \\
    R_{\Phi}(K,Y,R)\ll \frac{L}{T^{1/2}}+\log(b-a+1), \label{R_Phi_bound_new} \\
    b-a \ll v\Big(\big|\frac{F_{\ell}}{F}(\ell,\eta,\theta)\big|+\frac{1}{\ell}\Big)\ll \frac{T}{L}. \label{b_a_bound_new}
\end{gather}
Thus
\begin{equation} \label{weight_estimate_new}
    \frac{1}{\sqrt{\Phi_{\ell\ell}(u_0)}}\asymp \frac{L}{T^{1/2}}.
\end{equation}

The contribution of the error term \(R_{\Phi}\) to \(I_{K,Y,R,\Delta}^{(T)}\), using as usual $|a+b|^2 \le 2|a|^2 + 2|b|^2$, is
\[
    \Delta \int_T^{2T}
    \Big| \sum_{m,r} R_{\Phi}(K,Y,R)
    \Big|^2 dv \ll \Delta T (MR)^2
    \Big( \frac{L^2}{T}+\log^2(KT) \Big).
\]
Summing over dyadic \(T\) in the present range, we obtain
\begin{equation} \label{R_Phi_contribution_new}
    \sum_{\substack{T \text{ dyadic} \\ c_0 M \le T \ll \Delta^{-1}}} \Delta T (MR)^2
    \Big(
        \frac{L^2}{T}+ (\log KT)^2
    \Big) \ll 
    (KYR)^{\eps}
    \Big(
        \Delta K^2Y^{4\theta-2}R^2 +
        K^2Y^{2\theta-2}R^2
    \Big).
\end{equation}

Next, let
\[
    \Phi^{\ast}(m):=\Phi(u_0(m),m)-nu_0(m)+\frac{1}{8}.
\]
Differentiating the identity \(\Phi_{\ell}(u_0(m),m)=n\) with respect to \(m\), we get
\[
    \Phi_{\ell\ell}(u_0,m)\frac{du_0}{dm}+\Phi_{\ell m}(u_0,m)=0,
\]
and therefore
\begin{equation} \label{expression_for_u_0_new}
    \frac{du_0}{dm}=-\frac{\Phi_{\ell m}(u_0,m)}{\Phi_{\ell\ell}(u_0,m)}.
\end{equation}
Also,
\[
    \frac{d\Phi^{\ast}}{dm} =
    \Phi_{\ell}(u_0,m)\frac{du_0}{dm}+\Phi_m(u_0,m)-n\frac{du_0}{dm} = \Phi_m(u_0,m),
\]
hence
\[
    \frac{d^2\Phi^{\ast}}{dm^2} =
    \Phi_{mm}(u_0,m)-\frac{\Phi_{\ell m}(u_0,m)^2}{\Phi_{\ell\ell}(u_0,m)}.
\]

Differentiating~\eqref{def_Phi_new} twice with respect to \(m\), we obtain
\begin{equation} \label{Phi_mm_new}
    \Phi_{mm} =
    -v\Big(
        \frac{F_{mm}}{F}
        -\frac{F_m^2}{F^2}
        +\Big(1-\frac{1}{\theta}\Big)(\log \eta)_{mm}
    \Big).
\end{equation}
Using~\eqref{F_m_estimate_new} and~\eqref{F_mm_estimate_new}, we get
\[
    \frac{F_{mm}}{F}=O\Big(\frac{1}{m^2\log Y}\Big),
    \qquad \frac{F_m^2}{F^2}=O\Big(\frac{1}{m^2(\log Y)^2}\Big).
\]
Thus the main contribution to \(\Phi_{mm}\) comes from the \((\log \eta)_{mm}\)-term, and by~\eqref{estimate_for_log_eta_new},
\[
    \Phi_{mm}\asymp_{\theta} \frac{T}{m^2}.
\]

Next,
\[
    \Phi_{\ell m} = -v\Big(\frac{F_{\ell}}{F}\Big)_m =
    v\frac{F_{\ell}F_m}{F^2},
\]
because \(F_{\ell}\) does not depend on \(m\). Therefore
\[
    \Phi_{\ell m} = O\Big( \frac{T}{\ell m(\log Y)^2} \Big).
\]
Since \(\Phi_{\ell\ell}\asymp \frac{T}{\ell^2}\), we obtain
\[
    \frac{\Phi_{\ell m}^2}{\Phi_{\ell\ell}} =
    O\Big( \frac{T}{m^2(\log Y)^4} \Big),
\]
which is negligible compared with \(\Phi_{mm}\asymp \frac{T}{m^2}\). Hence
\[
    \frac{d^2\Phi^{\ast}}{dm^2} \asymp
    \frac{T}{m^2} \asymp \frac{T}{M^2}.
\]

Now we apply Lemma~\ref{Corput} with
\[
    \lambda_2:=\frac{T}{M^2}
\]
to the sum over \(m\). Then
\[
    \sum_{m\asymp M} e\big(\Phi^{\ast}(m)\big) \ll
    M\lambda_2^{\frac12}+\lambda_2^{-\frac12} \ll
    T^{\frac12}+MT^{-\frac12}.
\]
Using this together with~\eqref{IK_816_new}, \eqref{b_a_bound_new}, and~\eqref{weight_estimate_new}, we obtain
\begin{multline*}
    \Delta \int_T^{2T}
    \Big|
        \sum_{r\asymp R}
        \sum_{a<n<b}
        \frac{1}{(\Phi_{\ell\ell}(u_0))^{1/2}}
        \sum_{m\asymp M}
        e\big(\Phi^{\ast}(m)\big)
    \Big|^2 dv
    \ll \\
    \Delta T
    \Big(
        R\cdot \frac{T}{L}\cdot \frac{L}{T^{1/2}}
        \cdot \Big(T^{1/2}+\frac{M}{T^{1/2}}\Big)
    \Big)^2 \ll
    \Delta T R^2 (T+M)^2 \ll \Delta R^2(T^3+M^2T).
\end{multline*} Summing over dyadic \(T\) in the present range gives
\[
    \sum_{\substack{T \text{ dyadic} \\ c_0 M \le T\ll \Delta^{-1}}}
    \Delta T R^2(T+M)^2 \ll
    R^2(\Delta^{-2}+M^2).
\]

Under the current assumption,
$\Delta^{-2}\le \Delta K^2Y^{4\theta-2}$, and \(M^2=K^2Y^{2\theta-2}\). Thus
\begin{multline}
\label{main_after_corput_2_new}
    \sum_{\substack{T \text{ dyadic} \\ c_0 M \le T\ll \Delta^{-1}}}
    \Delta \int_T^{2T}
    \Big|
        \sum_{r\asymp R}
        \sum_{a<n<b}
        \frac{1}{(\Phi_{\ell\ell}(u_0))^{1/2}}
        \sum_{m\asymp M}
        e\big(\Phi^{\ast}(m)\big)
    \Big|^2 dv
    \ll \\
    \Delta K^2Y^{4\theta-2}R^2+K^2Y^{2\theta-2}R^2.
\end{multline}

Combining the bounds~\eqref{trivial_integral_new}, \eqref{Kusmin_Landau_integral_new}, \eqref{R_Phi_contribution_new}, and~\eqref{main_after_corput_2_new}, we find
\begin{equation} \label{final_bound_new}
    J_{K,Y,R,\Delta} \ll (KYR)^{\eps}
    \Big(
        \Delta K^2Y^{4\theta-2}R^2+K^2Y^{2\theta-2}R^2
    \Big).
\end{equation}

\vspace{2ex}

\textit{Case 2. \(\Delta < K^{-2/3}Y^{-4\theta/3+2/3}\).}

Set
\[
    \widetilde{\Delta}:=K^{-\frac23}Y^{-\frac{4\theta}{3}+\frac23}.
\]
By monotonicity and~\eqref{final_bound_new},
\begin{multline*}
    J_{K,Y,R,\Delta} \le
    J_{K,Y,R,\widetilde{\Delta}} \ll
    (KYR)^{\eps}
    \Big(
        \widetilde{\Delta} K^2Y^{4\theta-2}R^2
        +K^2Y^{2\theta-2}R^2
    \Big) = \\
    (KYR)^{\eps} \Big(
        K^{\frac{4}{3}}Y^{\frac{8\theta}{3}-\frac{4}{3}}R^2+K^2Y^{2\theta-2}R^2
    \Big).
\end{multline*}
Combining the two cases completes the proof.
\end{proof}

\begin{proposition} \label{count_small_theta_nondiag}
    Let \(K,Y_1,Y_2\ge 1\) be real numbers such that \(Y_2<2Y_1\), let \(0<\Delta\le 1\), \(0<\theta<\frac35\), and let \(0<c_1,c_2,c_3<1\) be fixed. Then, for every $\eps>0$, we have
    \begin{multline} \label{to_prove_small_theta}
        J_{K,Y_1,Y_2,\Delta} :=
        \# \Big\{
        c_1Y_1^{\theta}<\ell_1,\ell_2\le Y_1^{\theta},\
        c_2KY_2^{\theta-1}<m_1,m_2\le KY_2^{\theta-1}, \\
        c_3KY_1^{\theta-1}<n_1,n_2\le KY_1^{\theta-1},\ n_i\le \frac{m_i}{2}: \\
        \Big| F(\ell_1,\eta_1,\theta)\ell_1^{\frac1\theta}\eta_1^{1-\frac1\theta}
        -F(\ell_2,\eta_2,\theta)\ell_2^{\frac1\theta}\eta_2^{1-\frac1\theta}
        \Big| \le c_4 \Delta K Y_1^{\theta}\log Y_1
        \Big\} \ll \\
        (K Y_1 Y_2)^{\eps}
        \Big(
            K^4Y_1^{2\theta-2}Y_2^{2\theta-2}+
            K^{\frac{10}{3}}Y_1^{\frac{8\theta}{3}-\frac43}Y_2^{2\theta-2}+
            \Delta K^4Y_1^{4\theta-2}Y_2^{2\theta-2}
        \Big)
    \end{multline}
for any fixed $c_4>0$, where the implied constant may depend on $\theta,c_1,c_2,c_3,c_4$, and $\eps$,
where
    \begin{gather*}
        \eta_i:=\eta_i(m_i,n_i)=n_i^{-\frac{\theta}{1-\theta}}-m_i^{-\frac{\theta}{1-\theta}}, \\
        F(\ell,\eta,\theta) :=
        D(\theta)+\frac{1}{\theta^2}\log \ell-\frac{1}{\theta^2}\log \eta
        -\Big(1-\frac{1}{\theta}\Big)\frac{\eta_{\theta}}{\eta},
    \end{gather*}
    and \(D=D(\theta)\) is a constant.
\end{proposition}

\begin{proof}
The proof is very similar to that of Proposition~\ref{count_small_theta_diag}. As before, we split into two cases.

\vspace{2ex}

\textit{Case 1. \(\Delta \ge K^{-2/3}Y_1^{-4\theta/3+2/3}\).}

Here we may assume that \(KY_1^{\theta-1},KY_2^{\theta-1}\gg 1\), since otherwise the result is trivial. Applying the same steps as in the previous proposition, we obtain
\[
    J_{K,Y_1,Y_2,\Delta} \ll
    \Delta \int_0^{\Delta^{-1}}
    \Big| \sum_{\ell,m,n}
        e\Big(-v\big(
                \log F(\ell,\eta,\theta)
                +\frac{1}{\theta}\log \ell
                +\big(1-\frac{1}{\theta}\big)\log \eta
            \big)
        \Big)
    \Big|^2 dv.
\]
We then decompose the integral into
\[
    I_{K,Y_1,Y_2,\Delta}^{(0)} :=
    \Delta \int_0^1 |\ldots|^2 dv, \qquad
    I_{K,Y_1,Y_2,\Delta}^{(T)} :=
    \Delta \int_T^{2T} |\ldots|^2 dv,
    \quad T=1,2,4,\ldots.
\]
The integral over \([0,1]\) is bounded trivially:
\begin{equation} \label{trivial_integral_2}
    I_{K,Y_1,Y_2,\Delta}^{(0)} \ll
    \Delta Y_1^{2\theta}(K Y_1^{\theta-1})^2(K Y_2^{\theta-1})^2 =
    \Delta K^4Y_1^{4\theta-2}Y_2^{2\theta-2}.
\end{equation}

Next, introduce the phase
\begin{equation} \label{def_Phi_2}
    \Phi(\ell;v,\eta,\theta) :=
    -v\Big(
        \log F(\ell,\eta,\theta)
        +\frac{1}{\theta}\log \ell
        +\Big(1-\frac{1}{\theta}\Big)\log \eta
    \Big).
\end{equation}
Set
\[
    L:=Y_1^{\theta}, \qquad N:=K Y_1^{\theta-1}, \qquad M:=K Y_2^{\theta-1}.
\]

The derivative estimates needed below are proved exactly as in the proof of Proposition~\ref{count_small_theta_diag}, with \((m,m+r)\) replaced by \((n,m)\) and \(\frac{m}{m+r}\in [\frac{1}{101},1)\) replaced by \(\frac{n}{m}\in (0,\frac12]\). In particular, one has
\begin{equation} \label{eta_n_over_eta_estimate_2}
    \frac{\eta_n}{\eta}\asymp \frac{1}{n},
\end{equation}
\begin{equation} \label{estimate_for_log_eta_2}
    (\log \eta)_{nn}\asymp_{\theta} \frac{1}{n^2},
\end{equation}
and
\begin{equation} \label{F_n_estimate_2}
    F_n=O\Big(\frac{1}{n}\Big), \qquad
    F_{nn}=O\Big(\frac{1}{n^2}\Big).
\end{equation}

\vspace{2ex}

\textit{Range 1: \(1\le T\le c_0 N \).}

We again assume that $c_0>0$ is sufficiently small. As before, we have \(F(\ell,\eta,\theta)\gg_{\eps}\log Y_1\). Hence
\[
    \Phi_n(\ell;v,\eta,\theta) =
    -v\Big( \frac{F_n}{F} +
        \Big(1-\frac{1}{\theta}\Big)\frac{\eta_n}{\eta} \Big) \asymp \frac{T}{n}
    \asymp \frac{T}{N}.
\]
Since the derivative is monotone and has absolute value \(<1\) in this range, Lemma~\ref{Kusmin_Landau} applies to the sum over \(n\), giving
\[
    \sum_{n\asymp N} e\big(\Phi(\ell;v,\eta,\theta)\big)
    \ll \frac{N}{T}.
\]
Therefore
\[
    I_{K,Y_1,Y_2,\Delta}^{(T)} \ll
    \Delta T \Big( \sum_{\ell,m} \frac{N}{T}
    \Big)^2 \ll \Delta T
    \Big( Y_1^{\theta}\cdot K Y_2^{\theta-1}\cdot \frac{K Y_1^{\theta-1}}{T} \Big)^2.
\]
Summing over dyadic \(T\) in this range, we obtain
\begin{equation} \label{KL_contribution_2}
    \sum_{\substack{T \text{ dyadic} \\ 1\le T\le c_0 N }}
    I_{K,Y_1,Y_2,\Delta}^{(T)} \ll
    (K Y_1 Y_2)^{\eps}\Delta K^4Y_1^{4\theta-2}Y_2^{2\theta-2}.
\end{equation}

\vspace{1ex}

\textit{Range 2: \(c_0 N \le T\ll \Delta^{-1}\).}

Here we apply Poisson summation in \(\ell\), and then use Lemma~\ref{Corput} in \(n\). Since \(F(\ell,\eta,\theta)\gg_{\eps}\log Y_1\), one easily verifies
\[
    F_{\ell}\asymp \frac{1}{\ell}\asymp \frac{1}{Y_1^{\theta}}, \qquad
    F_{\ell\ell}\asymp -\frac{1}{\ell^2}\asymp -\frac{1}{Y_1^{2\theta}},
\]
and hence
\[
    \Phi_{\ell\ell}(\ell;v,\eta,\theta) =
    -v\Big(\frac{F_{\ell\ell}}{F}-\frac{F_{\ell}^2}{F^2}\Big)+\frac{v}{\theta \ell^2} \asymp \frac{T}{Y_1^{2\theta}}.
\]
Similarly,
\[
    \Phi_{\ell\ell\ell}\asymp \frac{T}{Y_1^{3\theta}}, \qquad
    \Phi_{\ell\ell\ell\ell}\asymp \frac{T}{Y_1^{4\theta}}.
\]

Applying~\cite[Theorem~8.16]{IK}, we obtain
\begin{equation} \label{IK_816_2}
    \sum_{\ell\asymp Y_1^{\theta}} e\big(\Phi(\ell;v,\eta,\theta)\big) =
    \sum_{a<\nu<b}
    \frac{1}{(\Phi_{\ell\ell}(u_0))^{1/2}}
    e\Big(\Phi(u_0)-\nu u_0+\frac{1}{8}\Big) +
    R_{\Phi}(K,Y_1,Y_2),
\end{equation}
where \(u_0=u_0(\nu,n,m,v)\),
\begin{gather}
    \Phi_{\ell}(u_0,n)=\nu, \label{def_nu_2} \\
    R_{\Phi}(K,Y_1,Y_2)\ll \frac{Y_1^{\theta}}{T^{1/2}}+\log(b-a+1), \label{R_Phi_bound_2} \\
    b-a \ll v\Big(\Big|\frac{F_{\ell}}{F}(\ell,\eta,\theta)\Big|+\frac{1}{\ell}\Big)\ll \frac{T}{Y_1^{\theta}}. \label{b_a_bound_2}
\end{gather}
Thus
\begin{equation} \label{weight_estimate_2}
    \frac{1}{\sqrt{\Phi_{\ell\ell}(u_0)}}\asymp \frac{Y_1^{\theta}}{T^{1/2}}.
\end{equation}

The contribution of the error term is
\[
    \Delta \int_T^{2T} \Big|
        \sum_{m,n} R_{\Phi}(K,Y_1,Y_2) \Big|^2 dv
    \ll \Delta T (N M)^2
    \Big(
        \frac{Y_1^{2\theta}}{T}+(\log KT)^2
    \Big).
\]
Summing over dyadic \(T\) in the present range, we obtain
\begin{multline} \label{R_Phi_contribution_2}
    \sum_{\substack{T \text{ dyadic} \\ c_0 N \le T\ll \Delta^{-1}}}
    \Delta \int_T^{2T}
    \Big|
        \sum_{m,n} R_{\Phi}(K,Y_1,Y_2)
    \Big|^2 dv
    \ll \\
    (K Y_1 Y_2)^{\eps}
    \Big(
        \Delta K^4Y_1^{4\theta-2}Y_2^{2\theta-2} +
        K^4Y_1^{2\theta-2}Y_2^{2\theta-2}
    \Big).
\end{multline}

Next, let
\[
    \Phi^{\ast}(n):=\Phi(u_0(n),n)-\nu u_0(n)+\frac{1}{8}.
\]
Differentiating the identity \(\Phi_{\ell}(u_0(n),n)=\nu\) with respect to \(n\), we get
\[
    \Phi_{\ell\ell}(u_0,n)\frac{du_0}{dn}+\Phi_{\ell n}(u_0,n)=0,
\]
and therefore
\[
    \frac{du_0}{dn}=-\frac{\Phi_{\ell n}(u_0,n)}{\Phi_{\ell\ell}(u_0,n)}.
\]
Also,
\[
    \frac{d\Phi^{\ast}}{dn} =
    \Phi_{\ell}(u_0,n)\frac{du_0}{dn}+\Phi_n(u_0,n)-\nu\frac{du_0}{dn} =
    \Phi_n(u_0,n),
\]
hence
\[
    \frac{d^2\Phi^{\ast}}{dn^2} =
    \Phi_{nn}(u_0,n)-\frac{\Phi_{\ell n}(u_0,n)^2}{\Phi_{\ell\ell}(u_0,n)}.
\]

Now
\[
    \Phi_{nn} =
    -v\Big(
        \frac{F_{nn}}{F}
        -\frac{F_n^2}{F^2}
        +\Big(1-\frac{1}{\theta}\Big)(\log \eta)_{nn}
    \Big),
\]
and by~\eqref{F_n_estimate_2} and~\eqref{estimate_for_log_eta_2} this implies
\[
    \Phi_{nn}\asymp_{\theta} \frac{T}{n^2}.
\]
Also,
\[
    \Phi_{\ell n} =
    -v\Big(\frac{F_{\ell}}{F}\Big)_n =
    v\frac{F_{\ell}F_n}{F^2},
\]
since \(F_{\ell}\) does not depend on \(n\). Therefore
\[
    \Phi_{\ell n} =
    O\Big( \frac{T}{\ell n(\log Y_1)^2} \Big).
\]
Since \(\Phi_{\ell\ell}\asymp \frac{T}{\ell^2}\), we obtain
\[
    \frac{\Phi_{\ell n}^2}{\Phi_{\ell\ell}} =
    O\Big( \frac{T}{n^2(\log Y_1)^4} \Big),
\]
which is negligible compared with \(\Phi_{nn}\asymp \frac{T}{n^2}\). Hence
\[
    \frac{d^2\Phi^{\ast}}{dn^2} \asymp
    \frac{T}{n^2} \asymp
    \frac{T}{(K Y_1^{\theta-1})^2}.
\]

Applying Lemma~\ref{Corput} with
\[
    \lambda_2:=\frac{T}{(K Y_1^{\theta-1})^2}
\]
to the sum over \(n\), we get
\[
    \sum_{n\asymp K Y_1^{\theta-1}} e\big(\Phi^{\ast}(n)\big) \ll
    K Y_1^{\theta-1}\lambda_2^{\frac12}+\lambda_2^{-\frac12}
    \ll T^{\frac12}+ K Y_1^{\theta-1} T^{-\frac12}.
\]
Therefore, using~\eqref{IK_816_2}, \eqref{b_a_bound_2}, and~\eqref{weight_estimate_2}, we get
\begin{multline*}
    \Delta \int_T^{2T}
    \Big|
        \sum_{m\asymp K Y_2^{\theta-1}}
        \sum_{a<\nu<b}
        \frac{1}{(\Phi_{\ell\ell}(u_0))^{1/2}}
        \sum_{n\asymp K Y_1^{\theta-1}}
        e\big(\Phi^{\ast}(n)\big)
    \Big|^2 dv
    \ll \\ \Delta T
    \Big(
        K Y_2^{\theta-1}
        \cdot \frac{T}{Y_1^{\theta}}
        \cdot \frac{Y_1^{\theta}}{T^{1/2}}
        \cdot \Big( T^{1/2}+\frac{K Y_1^{\theta-1}}{T^{1/2}} \Big)
    \Big)^2 \ll
    \Delta T (K Y_2^{\theta-1})^2 \big( T+K Y_1^{\theta-1} \big)^2.
\end{multline*}

Summing over dyadic \(T\) gives
\[
    \sum_{\substack{T \text{ dyadic} \\ c_0 N \le T\ll \Delta^{-1}}}
    \Delta \int_T^{2T} |\ldots|^2\,dv \ll
    (K Y_2^{\theta-1})^2
    \big(
        \Delta^{-2}+(K Y_1^{\theta-1})^2
    \big).
\]
By the current assumption,
\[
    \Delta^{-2}\le \Delta K^2Y_1^{4\theta-2},
\]
and hence
\[
    (K Y_2^{\theta-1})^2
    \big(
        \Delta^{-2}+(K Y_1^{\theta-1})^2
    \big)
    \ll \Delta K^4Y_1^{4\theta-2}Y_2^{2\theta-2}
    + K^4Y_1^{2\theta-2}Y_2^{2\theta-2}.
\]

Collecting together~\eqref{trivial_integral_2}, \eqref{KL_contribution_2}, \eqref{R_Phi_contribution_2}, and the last estimate, we find
\begin{equation} \label{final_bound_2}
    J_{K,Y_1,Y_2,\Delta} \ll
    (K Y_1 Y_2)^{\eps}
    \Big(
        \Delta K^4Y_1^{4\theta-2}Y_2^{2\theta-2} +
        K^4Y_1^{2\theta-2}Y_2^{2\theta-2}
    \Big).
\end{equation}

\vspace{2ex}

\textit{Case 2. \(\Delta < K^{-2/3}Y_1^{-4\theta/3+2/3}\).}

Here we again use monotonicity, with
\[
    \widetilde{\Delta}:=K^{-\frac23} Y_1^{-\frac{4\theta}{3}+\frac23}.
\]
Then from~\eqref{final_bound_2} we get
\begin{multline*}
    J_{K,Y_1,Y_2,\Delta} \ll
    (K Y_1 Y_2)^{\eps}
    \Big(
        \widetilde{\Delta} K^4Y_1^{4\theta-2}Y_2^{2\theta-2} +
        K^4Y_1^{2\theta-2}Y_2^{2\theta-2}
    \Big) \ll \\
    (K Y_1 Y_2)^{\eps}
    \Big(
        K^4Y_1^{2\theta-2}Y_2^{2\theta-2} +
        K^{\frac{10}{3}}Y_1^{\frac{8\theta}{3}-\frac{4}{3}}Y_2^{2\theta-2}
    \Big),
\end{multline*}
which completes the proof.
\end{proof}

\bibliographystyle{abbrv}
\bibliography{Bib}

\end{document}